\documentclass[12pt]{article}
\usepackage[a4paper, total={17cm,25cm}]{geometry}

\usepackage{amsmath,amssymb,amsfonts}
\usepackage[alphabetic,y2k,initials]{amsrefs}

\usepackage{color, graphics}
\usepackage{float}
\usepackage{tikz}
\usepackage{pdfsync}
\usepackage{hyperref}
\usepackage{multirow}

\usepackage{longtable}

\usepackage{calrsfs}
\DeclareMathAlphabet{\pazocal}{OMS}{zplm}{m}{n}

\newcommand{\C}{\pazocal{C}}
\newcommand{\Q}{\pazocal{Q}}
\newcommand{\E}{\pazocal{E}}

\newcommand{\Curve}{\mathcal{C}}

\newcommand{\refeq}{\eqref}

\def\sign{\mathrm{sign\,}}

\def\dist{\mathrm{dist}}
\def\sn{\,\mathrm{sn}}
\def\cn{\,\mathrm{cn}}
\def\dn{\,\mathrm{dn}}
\def\arccot{\mathrm{arccot}}

\newtheorem{theorem}{Theorem}[section]
\newtheorem{proposition}[theorem]{Proposition}
\newtheorem{corollary}[theorem]{Corollary}

\newtheorem{remark}[theorem]{Remark}
\newtheorem{example}[theorem]{Example}
\newtheorem{definition}[theorem]{Definition}

\newenvironment{proof}[1][Proof]{\noindent\textit{#1.} }{\hfill$\Box$\newline\medskip}

\numberwithin{equation}{section}

\usepackage{authblk}
\author[1]{Anani Komla Adabrah}
\author[1,3]{Vladimir Dragovi\'c}
\author[2,3]{Milena Radnovi\'c}
\affil[1]{\textsc{The University of Texas at Dallas, Department of Mathematical Sciences}}
\affil[2]{\textsc{The University of Sydney, School of Mathematics and Statistics}}
\affil[3]{\textsc{Mathematical Institute SANU, Belgrade}}
\affil[ ]{\texttt{ananikomla.adabrah@utdallas.edu, vladimir.dragovic@utdallas.edu, milena.radnovic@sydney.edu.au}}

\date{}

\title{Periodic billiards within conics in the Minkowski plane and Akhiezer polynomials}

\begin{document}
	
	\maketitle
	
\begin{abstract}
We derive necessary and sufficient conditions for periodic and for elliptic periodic trajectories of billiards within an ellipse in the Minkowski plane in terms of an underlining elliptic curve.
We provide several examples of periodic and elliptic periodic trajectories with small periods.
We observe relationship between Cayley-type conditions and discriminantly separable and factorizable polynomials.
Equivalent conditions for periodicity and elliptic periodicity are derived in terms of polynomial-functional equations as well. The corresponding polynomials are related to the classical extremal polynomials.
In particular, the light-like periodic trajectories are related to the classical Chebyshev polynomials.
The similarities and differences with respect to previously studied Euclidean case are indicated.

\

\noindent
\textsc{MSC2010 numbers}: \texttt{14H70, 41A10, 70H06, 37J35, 26C05}

\end{abstract}

\tableofcontents

\section{Introduction}\label{sec:intro}

Billiards within quadrics in pseudo-Euclidean spaces were studied in \cites{KhTab2009,DragRadn2012adv,DragRadn2013publ}.
In \cites{DragRadn2018,DragRadn2019rcd}, the relationship between the billiards within quadrics in the Euclidean spaces and extremal polynomials has been studied.
The aim of this paper is to develop the connection between extremal polynomials and billiards in the Minkowski plane.

This paper is organised as follows.
In Section \ref{sec:confocal}, we recall the basic notions connected with the Minkowski plane, confocal families of conics, relativistic ellipses and hyperbolas, and billiards.
In Section \ref{sec:periodic}, we give a complete description of the periodic billiard trajectories in algebro-geometric terms.
In Section \ref{sec:small}, we use the conditions obtained in the previous section to study examples of periodic trajectories with small periods.

We also emphasize intriguing connection between the Cayley-type conditions and discriminantly separable polynomials.
The notion of relativistic ellipses and hyperbolas enables definition of Jacobi-type elliptic coordinates in the Minkowski setting.
Since the correspondence between Cartesian and elliptic coordinates is not one-to-one, there is a notion of elliptic periodicity which refers to a weaker assumption that a trajectory is periodic in elliptic coordinates.
In Section \ref{sec:elliptic}, we provide algebro-geometric characterization of trajectories to be $n$-elliptic periodic without being $n$-periodic.
Section \ref{sec:examples-elliptic} provides examples and connections with discriminantly separable polynomials.
In Section \ref{sec:polynomial}, we derive a characterisation of elliptic periodic trajectories using polynomial equations.
In the last Section \ref{sec:extremal}, we establish the connection between characteristics of periodic billiard trajectories and extremal polynomials:
the Zolotarev polynomials, the Akhiezer polynomials on symmetric intervals, and the general Akhiezer polynomials on two intervals.
We conclude our study of the relationship of billiards in the Minkowski plane with the extremal polynomials by relating the case of light-like trajectories to the classical Chebyshev polynomials, see Section \ref{sec:ll}.

Apart from similarities with previously studied Euclidean spaces, see \cite{DragRadn2018, DragRadn2019rcd}, there are also significant differences:
for example, among the obtained extremal polynomials are such with winding numbers $(3,1)$, which was never the case in the Euclidean setting.

\section{Confocal families of conics and billiards}
\label{sec:confocal}

\paragraph*{The Minkowski plane} is $\mathbf{R}^2$ with \emph{the Minkowski scalar product}: $\langle X,Y\rangle=X_1Y_1-X_2Y_2$.

\emph{The Minkowski distance} between points $X$, $Y$ is
$
\dist(X,Y)=\sqrt{\langle{X-Y,X-Y}\rangle}.
$
Since the scalar product can be negative, notice that the Minkowski distance can have imaginary values as well.
In that case, we choose the value of the square root with the positive imaginary part.

Let $\ell$ be a line in the Minkowski plane, and $v$ its vector.
$\ell$ is called
\emph{space-like} if $\langle{v,v}\rangle>0$;
\emph{time-like} if $\langle{v,v}\rangle<0$;
and \emph{light-like} if $\langle{v,v}\rangle=0$.
Two vectors $x$, $y$ are \emph{orthogonal} in the Minkowski plane if $\langle x,y \rangle=0$.
Note that a light-like vector is orthogonal to itself.

\paragraph*{Confocal families.}
Denote by
\begin{equation}\label{eq:ellipse}
\E\ :\ \frac{\mathsf{x}^2}{a}+\frac{\mathsf{y}^2}{b}=1
\end{equation}
an ellipse in the plane, with $a$, $b$ being fixed positive numbers.

The associated family of confocal conics is:
\begin{equation}\label{eq:confocal.conics}
\C_{\lambda}\ :\
\frac{\mathsf{x}^2}{a-\lambda}+\frac{\mathsf{y}^2}{b+\lambda}=1, \quad
\lambda\in\mathbf{R}.
\end{equation}

The family is shown on Figure \ref{fig:confocal.conics}.
\begin{figure}[h]
	\begin{center}
\begin{tikzpicture}[scale=1.5]



\draw[thick,dashed](.7,0) arc (0:60:0.7cm and 1.2cm);
\draw[thick](.35,1.05) arc (60:120:0.7cm and 1.2cm);
\draw[thick,dashed](-.35,1.05) arc (120:180:0.7cm and 1.2cm);
\draw[thick,dashed](-.7,0) arc (180:240:0.7cm and 1.2cm);
\draw[thick](-.35,-1.05) arc (240:300:0.7cm and 1.2cm);
\draw[thick,dashed](.35,-1.05) arc (300:360:0.7cm and 1.2cm);


\draw[thick,dashed](1.325,-.01) arc (0:30:1.32287565553cm and .5cm);
\draw[thick](1.14,.25) arc (30:150:1.32287565553cm and .5cm);
\draw[thick,dashed](-1.14,.25) arc (150:210:1.32287565553cm and .5cm);
\draw[thick](-1.14,-.25) arc (210:330:1.32287565553cm and .5cm);
\draw[thick,dashed](1.14,-.25) arc (330:360:1.32287565553cm and .5cm);

\draw[thick,dashed](1.13,0) arc (0:30:1.1401754251cm and 0.83666002653 cm); 
\draw[thick](.99,.41) arc (30:150:1.1401754251cm and 0.83666002653 cm);
\draw[thick,dashed](-.99,.41) arc (150:210:1.1401754251cm and 0.83666002653 cm);
\draw[thick](-.99,-.41) arc (210:330:1.1401754251cm and 0.83666002653 cm);
\draw[thick,dashed](.99,-.41) arc (330:357:1.1401754251cm and 0.83666002653 cm);

\draw[gray] (2.61,-1.2) -- (-1.2,2.61);
\draw[gray] (2.61,1.2) -- (-1.2,-2.61);
\draw[gray] (-2.61,1.2) -- (1.2,-2.61);
\draw[gray] (-2.61,-1.2) -- (1.2,2.61);

\draw[domain=-1.5:-.9,smooth,thick,variable=\t] plot ({\t},{sqrt(1+2.5)*sqrt(1-\t*\t/(1-2.5))});
\draw[domain=-.9:.9,smooth,thick,dashed,variable=\t] plot ({\t},{sqrt(1+2.5)*sqrt(1-\t*\t/(1-2.5))});
\draw[domain=.9:1.5,smooth,thick,variable=\t] plot ({\t},{sqrt(1+2.5)*sqrt(1-\t*\t/(1-2.5))});

\draw[domain=-1.5:-.9,smooth,thick,variable=\t] plot ({\t},{-sqrt(1+2.5)*sqrt(1-\t*\t/(1-2.5))});
\draw[domain=-.9:.9,smooth,thick,dashed,variable=\t] plot ({\t},{-sqrt(1+2.5)*sqrt(1-\t*\t/(1-2.5))});
\draw[domain=.9:1.5,smooth,thick,variable=\t] plot ({\t},{-sqrt(1+2.5)*sqrt(1-\t*\t/(1-2.5))});

\draw[domain=-1.2:-.5,smooth,thick,variable=\t] plot ({\t},{sqrt(1+1.5)*sqrt(1-\t*\t/(1-1.5))});
\draw[domain=-.5:.5,smooth,thick,dashed,variable=\t] plot ({\t},{sqrt(1+1.5)*sqrt(1-\t*\t/(1-1.5))});
\draw[domain=.5:1.2,smooth,thick,variable=\t] plot ({\t},{sqrt(1+1.5)*sqrt(1-\t*\t/(1-1.5))});

\draw[domain=-1.2:-.5,smooth,thick,variable=\t] plot ({\t},{-sqrt(1+1.5)*sqrt(1-\t*\t/(1-1.5))});
\draw[domain=-.5:.5,smooth,thick,dashed,variable=\t] plot ({\t},{-sqrt(1+1.5)*sqrt(1-\t*\t/(1-1.5))});
\draw[domain=.5:1.2,smooth,thick,variable=\t] plot ({\t},{-sqrt(1+1.5)*sqrt(1-\t*\t/(1-1.5))});

\draw[domain=-1.5:-.9,smooth,thick,dashed,variable=\t] plot ({sqrt(1+2.5)*sqrt(1-\t*\t/(1-2.5))},{\t});
\draw[domain=-.9:.9,smooth,thick,variable=\t] plot ({sqrt(1+2.5)*sqrt(1-\t*\t/(1-2.5))},{\t});
\draw[domain=.9:1.5,smooth,thick,dashed,variable=\t] plot ({sqrt(1+2.5)*sqrt(1-\t*\t/(1-2.5))},{\t});

\draw[domain=-1.5:-.9,smooth,thick,,dashed,variable=\t] plot ({-sqrt(1+2.5)*sqrt(1-\t*\t/(1-2.5))},{\t});
\draw[domain=-.9:.9,smooth,thick,variable=\t] plot ({-sqrt(1+2.5)*sqrt(1-\t*\t/(1-2.5))},{\t});
\draw[domain=.9:1.5,smooth,thick,,dashed,variable=\t] plot ({-sqrt(1+2.5)*sqrt(1-\t*\t/(1-2.5))},{\t});

\draw[domain=-1.2:-.4,smooth,thick,dashed,variable=\t] plot ({sqrt(1+1.5)*sqrt(1-\t*\t/(1-1.5))},{\t});
\draw[domain=-.4:.4,smooth,thick,variable=\t] plot ({sqrt(1+1.5)*sqrt(1-\t*\t/(1-1.5))},{\t});
\draw[domain=.4:1.2,smooth,thick,dashed,variable=\t] plot ({sqrt(1+1.5)*sqrt(1-\t*\t/(1-1.5))},{\t});

\draw[domain=-1.2:-.4,smooth,thick,,dashed,variable=\t] plot ({-sqrt(1+1.5)*sqrt(1-\t*\t/(1-1.5))},{\t});
\draw[domain=-.4:.4,smooth,thick,variable=\t] plot ({-sqrt(1+1.5)*sqrt(1-\t*\t/(1-1.5))},{\t});
\draw[domain=.4:1.2,smooth,thick,,dashed,variable=\t] plot ({-sqrt(1+1.5)*sqrt(1-\t*\t/(1-1.5))},{\t});

\end{tikzpicture}
		
	\end{center}
	\caption{Family of confocal conics in the Minkowski
		plane. Solid lines represent relativistic ellipses, and dashed ones relativistic hyperbolas.}\label{fig:confocal.conics}
\end{figure}
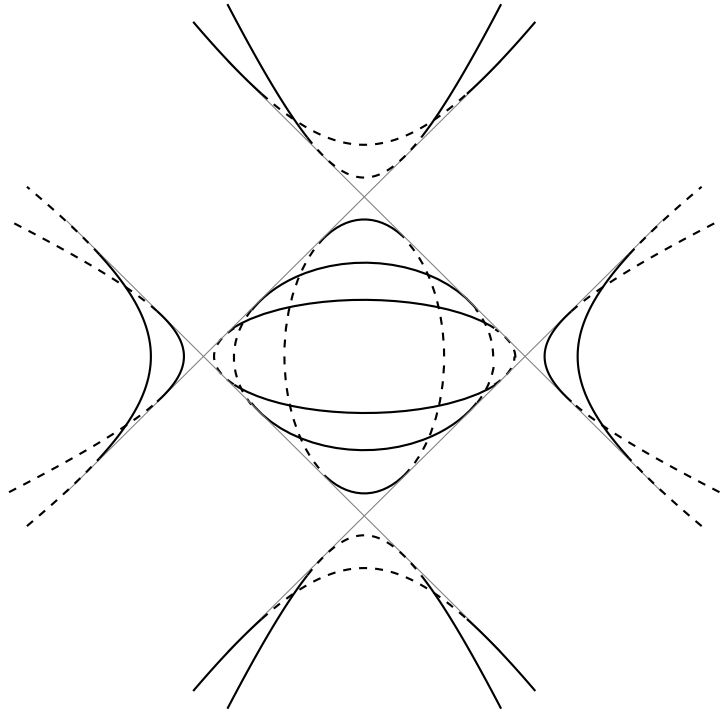
We may distinguish the following three subfamilies in the family
\refeq{eq:confocal.conics}:
for $\lambda\in(-b,a)$, conic $\C_{\lambda}$ is an ellipse;
for $\lambda<-b$, conic $\C_{\lambda}$ is a hyperbola with $\mathsf{x}$-axis as the major one;
for $\lambda>a$, it is a hyperbola again, but now its major axis is $\mathsf{y}$-axis.
In addition, there are three degenerated quadrics: $\C_{a}$, $\C_{b}$, $\C_{\infty}$ corresponding to $\mathsf{y}$-axis, $\mathsf{x}$-axis, and the line at the infinity respectively.

The confocal family has three pairs of foci:
$F_1(\sqrt{a+b},0)$, $F_2(-\sqrt{a+b},0)$;
$G_1(0,\sqrt{a+b})$, $G_2(0,-\sqrt{a+b})$; and
$H_1(1:-1:0)$, $H_2(1:1:0)$ on the line at the infinity.

We notice four distinguished lines:
\begin{align*}
&\mathsf{x}+\mathsf{y}=\sqrt{a+b},\quad \mathsf{x}+\mathsf{y}=-\sqrt{a+b},\\
&\mathsf{x}-\mathsf{y}=\sqrt{a+b},\quad \mathsf{x}-\mathsf{y}=-\sqrt{a+b}.
\end{align*}
These lines
are common tangents to all conics from the family.

Conics in the Minkowski plane have geometric properties analogous to the conics in the Euclidean plane.
Namely, for each point on conic $\C_{\lambda}$, either sum or difference of its Minkowski distances from the foci $F_1$ and $F_2$ is equal to $2\sqrt{a-\lambda}$;
either sum or difference of the distances from the other pair of foci $G_1$, $G_2$ is equal to $2\sqrt{-b-\lambda}$ \cite{DragRadn2012adv}.

In the Minkowkski plane, it is natural to consider relativistic conics, which are suggested in \cite{BirkM1962}.
In this section, we give a brief account of the related analysis.

Consider points $F_1(\sqrt{a+b},0)$ and $F_2(-\sqrt{a+b},0)$.

For a given constant $c\in\mathbf{R}^{+}\cup i\mathbf{R}^{+}$, \emph{a relativistic ellipse} is the set of points $X$ satisfying
$
\dist(F_1,X)+\dist(F_2,X)=2c,
$
while \emph{a relativistic hyperbola} is the union of the sets given by the following equations:
\begin{gather*}
\dist(F_1,X)-\dist(F_2,X)=2c,\\
\dist(F_2,X)-\dist(F_1,X)=2c.
\end{gather*}

Relativistic conics can be described as follows.
\begin{description}
	\item[$0<c<\sqrt{a+b}$]
	The corresponding relativistic conics lie on ellipse $\C_{a-c^2}$ from family \refeq{eq:confocal.conics}.
	The ellipse $\C_{a-c^2}$ is split into four arcs by touching points with the four common tangent lines; thus, the relativistic ellipse is the union of the two arcs intersecting the $\mathsf{y}$-axis, while the relativistic hyperbola is the union of the other two arcs.
	
	\item[$c>\sqrt{a+b}$]
	The relativistic conics lie on $\C_{a-c^2}$ -- a hyperbola with $\mathsf{x}$-axis as the major one.
	Each branch of the hyperbola is split into three arcs by touching points with the common tangents; thus, the relativistic ellipse is the union of the two finite arcs, while the relativistic hyperbola is the union of the four infinite ones.
	
	\item[$c$ is imaginary]
	The relativistic conics lie on hyperbola $\C_{a-c^2}$ -- a hyperbola with $\mathsf{y}$-axis as the major one.
	As in the previous case, the branches are split into six arcs in total by common points with the four tangents.
	The relativistic ellipse is the union of the four infinite arcs, while the relativistic hyperbola is the union of the two finite ones.
\end{description}

Notice that all relativistic ellipses are disjoint with each other, as well as all relativistic hyperbolas, see Figure \ref{fig:confocal.conics}.
Moreover, at the intersection point of a relativistic ellipse which is a part of the geometric conic $\C_{\lambda_1}$ from the confocal family \refeq{eq:confocal.conics} and a relativistic hyperbola belonging to $\C_{\lambda_2}$, it is always $\lambda_1<\lambda_2$.

\paragraph*{Elliptic coordinates.}
Each point inside ellipse $\E$ has elliptic coordinates $(\lambda_1,\lambda_2)$, such that $-b<\lambda_1<0<\lambda_2<a$.

The differential equation of the lines touching a given conic $\C_{\gamma}$ is:
\begin{equation}\label{eq:diff-eq}
\frac{d\lambda_1}{\sqrt{(a-\lambda_1)(b+\lambda_1)(\gamma-\lambda_1)}}
+
\frac{d\lambda_2}{\sqrt{(a-\lambda_2)(b+\lambda_2)(\gamma-\lambda_2)}}=0.
\end{equation}

\paragraph*{Billiards.}

Let $v$ be a vector and $p$ a line in the Minkowski plane.
Decompose vector $v$ into the sum $v=a+n_{p}$ of a vector $n_{p}$ orthogonal to $p$ and $a$ belonging to $p$.
Then vector $v'=a-n_{p}$ is \emph{the billiard reflection} of $v$ on $p$.
It is easy to see that $v$ is also the billiard reflection of $v'$ with respect to $p$.
Moreover, since $\langle{v,v}\rangle=\langle{v',v'}\rangle$, vectors $v$, $v'$ are of the same type.


Note that $v=v'$ if $v$ is contained in $p$ and $v'=-v$ if it is orthogonal to $p$.
If $n_{p}$ is light-like, which means that it belongs to $p$, then the reflection is not defined.

Line $\ell'$ is \emph{the billiard reflection} of $\ell$ off ellipse $\E$ if their intersection point $\ell\cap\ell'$ belongs to $\E$ and the vectors of $\ell$, $\ell'$ are reflections of each other with respect to the tangent line of $\E$ at this point.

The lines containing segments of a given billiard trajectory within $\E$ are all of the same type: they are all either space-like, time-like, or light-like.
For the detailed explanation, see \cite{KhTab2009}.

Billiard trajectories within ellipses in the Minkowski plane have caustic properties: each segment of a given trajectory will be tangent to the same conic confocal with the boundary, see \cite{DragRadn2012adv}.
More about Minkowski plane or higher-dimensional pseudo-Euclidean spaces and related integrable systems can be found in \cites{BirkM1962,GKT2007,WFSWZZ2009, JJ1, JJ2}.

\section{Periodic trajectories}\label{sec:periodic}

Sections \ref{sec:periodic}--\ref{sec:extremal} deal with the trajectories with non-degenerate caustic $\C_{\gamma}$, which will mean that $\gamma\in\mathbf{R}\setminus\{-b,a\}$.
Such trajectories are either space-like or time-like.
The case of light-like trajectories, which correspond to the degenerate caustic $\C_{\infty}$ is considered separately, in Section \ref{sec:ll}.

The periodic trajectories of elliptical billiards in the Minkowski plane can be characterized in algebro-geometric terms using the underlying elliptic curve:

\begin{theorem}\label{th:curve-billiard}
	The billiard trajectories within $\E$ with non-degenerate caustic $\C_{\gamma}$ are $n$-periodic if and only if  $nQ_{0}\sim nQ_{\gamma}$ on the elliptic curve:
	\begin{equation}\label{eq:billiard-cubic}
	\Curve\ :\ y^2=\varepsilon(a-x)(b+x)(\gamma-x),
	\end{equation}
	with $Q_0$ being a point of $\Curve$ corresponding to $x=0$, $Q_{\gamma}$ to $x=\gamma$, and $\varepsilon=\sign\gamma$.
\end{theorem}

\begin{proof}
Along a billiard trajectory within $\E$ with caustic $\C_{\gamma}$, the elliptic coordinate $\lambda_1$ traces the segment $[\alpha_1,0]$, and $\lambda_2$ the segment $[0,\beta_1]$, where $\alpha_1$ is the largest negative and $\beta_1$ the smallest positive member of the set $\{a,-b,\gamma\}$.

\emph{Case 1.}
If $\C_{\gamma}$ is an ellipse and $\gamma<0$, then $\alpha_1=\gamma$, $\beta_1=a$.
The coordinate $\lambda_1$ takes value $\lambda_1=\gamma$ at the touching points with the caustic and value $\lambda_1=0$ at the reflection points off the arcs of $\E$ where the restricted metric is time-like.
On the other hand, $\lambda_2$ takes value $\lambda_2=a$ at the intersections with $\mathsf{y}$-axis, and $\lambda_2=0$ at the reflection points off the arcs of $\E$ where the restricted metric is space-like.

\emph{Case 2.}
If $\C_{\gamma}$ is an ellipse and $\gamma>0$, then $\alpha_1=-b$, $\beta_1=\gamma$.
The coordinate $\lambda_1$ takes value $\lambda_1=-b$ at the intersections with $\mathsf{x}$-axis and value $\lambda_1=0$ at the reflection points off the arcs of $\E$ where the restricted metric is time-like.
On the other hand, $\lambda_2$ takes value $\lambda_2=\gamma$ at the touching points with the caustic, and $\lambda_2=0$ at the reflection points off the arcs of $\E$ where the restrictes metric is space-like.

\emph{Case 3.}
If $\C_{\gamma}$ is a hyperbola, then $\alpha_1=-b$, $\beta_1=a$.
The coordinate $\lambda_1$ takes value $\lambda_1=-b$ at the intersections with $\mathsf{x}$-axis and value $\lambda_1=0$ at the reflection points off the arcs of $\E$ where the restricted metric is time-like.
On the other hand, $\lambda_2$ takes value $\lambda_2=a$ at the intersections with $\mathsf{y}$-axis, and $\lambda_2=0$ at the reflection points off the arcs of $\E$ where the restricted metric is space-like.

In each case, the elliptic coordinates change monotonously between their extreme values.

Consider an $n$-periodic billiard trajectory and denote by $n_1$ the number of reflections off time-like arcs, i.e.~off relativistic ellipses, and by $n_2$ the number of reflections off space-like arcs, i.e.~relativistic hyperbolas.
Obviously, $n_1+n_2=n$.
Integrating \eqref{eq:diff-eq} along the trajectory, we get:
\begin{equation}\label{InteEqn}
n_1\int_{\alpha_1}^{0}\frac{d\lambda_1}{\sqrt{\varepsilon(a-\lambda_1)(b+\lambda_1)(\gamma-\lambda_1)}}
+
n_2\int_{\beta_1}^{0}\frac{d\lambda_2}{\sqrt{\varepsilon(a-\lambda_2)(b+\lambda_2)(\gamma-\lambda_2)}}
=0,
\end{equation}
i.e.
$$
n_1(Q_0-Q_{\alpha_1})+n_2(Q_0-Q_{\beta_1})\sim0.
$$
In Case 1, this is equivalent to
$$
n_1(Q_0-Q_{\gamma})+n_2(Q_0-Q_{a})\sim n(Q_0-Q_{\gamma}),
$$
since a closed trajectory crosses the $\mathsf{y}$-axis even number of times, i.e~$n_2$ must be even, and $2Q_{a}\sim2Q_{\gamma}$.

Similarly, in Case 2, it follows since $n_1$ is even, and in Case 3 both $n_1$ and $n_2$ need to be even.
\end{proof}

From the proof of Theorem \ref{th:curve-billiard}, we have:

\begin{corollary}
The period of a closed trajectory with hyperbola as caustic is even.
\end{corollary}

\begin{theorem}\label{th:cayley-billiard}
The billiard trajectories within $\E$ with caustic $\C_{\gamma}$ are $n$-periodic if and only if:
	\begin{gather*}
	C_2=0,
	\quad
	\left|
	\begin{array}{cc}
	C_2 & C_3
	\\
	C_3 & C_4
	\end{array}
	\right|=0,
	\quad
	\left|
	\begin{array}{ccc}
	C_2 & C_3 & C_4
	\\
	C_3 & C_4 & C_5
	\\
	C_4 & C_5 & C_6
	\end{array}
	\right|=0,
	\dots
	\quad\text{for}\quad n=3,5,7,\dots
	\\
	B_3=0,
	\quad
	\left|
	\begin{array}{cc}
	B_3 & B_4
	\\
	B_4 & B_5
	\end{array}
	\right|=0,
	\quad
	\left|
	\begin{array}{ccc}
	B_3 & B_4 & B_5
	\\
	B_4 & B_5 & B_6
	\\
	B_5 & B_6 & B_7
	\end{array}
	\right|=0,
	\dots
	\quad\text{for}\quad n=4,6,8,\dots.
	\end{gather*}
	Here, we denoted:
	\begin{gather*}
	\sqrt{\varepsilon(a-x)(b+x)(\gamma-x)}=B_0+B_1x+B_2x^2+\dots,
	\\
	\frac{\sqrt{\varepsilon(a-x)(b+x)(\gamma-x)}}{\gamma-x}=C_0+C_1x+C_2x^2+\dots,
	\end{gather*}
	the Taylor expansions around $x=0$.
\end{theorem}
\begin{proof}
	Denote by $Q_{\infty}$ the point of $\Curve$ \refeq{eq:billiard-cubic} corresponding to $x=\infty$ and notice that
	\begin{equation}\label{eq:2Q}
	2Q_{\gamma}\sim2 Q_{\infty}.
	\end{equation}
	
	Consider first $n$ even.
	Because of \refeq{eq:2Q}, the condition $nQ_{0}\sim nQ_{\gamma}$ is equivalent to $nQ_{0}\sim nQ_{\infty}$, which is equivalent to the existence of a meromorphic function of $\Curve$ with the unique pole at $Q_{\infty}$ and unique zero at $Q_{0}$, such that the pole and the zero are both of multiplicity $n$.
	The basis of $\mathcal{L}(nQ_{\infty})$ is:
	\begin{equation}\label{eq:basis-even}
	1,x,x^2,\dots,x^{n/2},y,xy, x^{n/2-2}y,
	\end{equation}
	thus a non-trivial linear combination of those functions with a zero of order $n$ at $x=0$ exists if and only if:
	$$
	\left|
	\begin{array}{llll}
	B_{n/2+1} & B_{n/2} & \dots & B_3\\
	B_{n/2+2} & B_{n/2+1} &\dots & B_4\\
	\dots\\
	B_{n-1} & B_n &\dots & B_{n/2+1}
	\end{array}
	\right|=0.
	$$
	
	Now, suppose $n$ is odd.
	Because of \refeq{eq:2Q}, the condition $nQ_{0}\sim nQ_{\gamma}$ is equivalent to $nQ_{0}\sim (n-1)Q_{\infty}+Q_{\gamma}$, which is equivalent to the existence of a meromorphic function of $\Curve$ with only two poles: of order $n-1$ at $Q_{\infty}$ and a simple pole at $Q_{\gamma}$, and unique zero at $Q_{0}$.
	The basis $\mathcal{L}( (n-1)Q_{\infty}+Q_{\gamma})$ is:
	\begin{equation}\label{eq:basis-odd}
	1,x,x^2,\dots,x^{(n-1)/2},\frac{y}{\gamma-x}, \frac{xy}{\gamma-x}, \dots, \frac{x^{(n-1)/2-1}y}{\gamma-x},
	\end{equation}
	thus a non-trivial linear combination of those functions with a zero of order $n$ at $x=0$ exists if and only if:
	$$
	\left|
	\begin{array}{llll}
	C_{(n-1)/2+1} & C_{(n-1)/2} & \dots & C_2\\
	C_{(n-1)/2+2} & C_{(n-1)/2+1} &\dots & C_3\\
	\dots\\
	C_{n-1} & C_n &\dots & C_{(n-1)/2+1}
	\end{array}
	\right|=0.
	$$
\end{proof}

\section{Trajectories with small periods and discriminantly separable polynomials}
\label{sec:small}

\subsection{Examples of periodic trajectories: $3\le n\le8$}\label{sec:examples-table}

\subsubsection*{3-periodic trajectories}
There is a $3$-periodic trajectory of the billiard within \refeq{eq:ellipse}, with a non-degenerate caustic $\C_{\gamma}$ in the Minkowski plane if and only if, according to Theorem \refeq{th:cayley-billiard}, the caustic is an ellipse, i.e.~$\gamma \in (-b,a)$ and  $C_2=0$.

We solve the equation
\begin{equation}\label{eqn:3-periodic}
C_2
=
\dfrac{3a^{2}b^{2}+2ab(a-b)\gamma-(a+b)^{2}\gamma^{2} }
{8(ab)^{3/2}\gamma^{5/2}}=0,
\end{equation}

which yields the following two solutions for the parameter $\gamma$ for the caustic:
\begin{equation}\label{CaleyEq1}
{\gamma}_{1}= \dfrac{ab}{(a+b)^{2}}(a-b+2\sqrt{a^{2}+ab+b^{2}}),
\quad
{\gamma}_{2}= -\dfrac{ab}{(a+b)^{2}}(-a+b+2\sqrt{a^{2}+ab+b^{2}}).
\end{equation}
Notice that both caustics $\C_{{\gamma}_{2}}$ and $\C_{{\gamma}_{1}}$ are ellipses since $-b<{\gamma}_{2}<0<{\gamma}_{1}<a$.

Two examples of a $3$-periodic trajectories are shown in Figure \ref{fig:3-periodic}.
\begin{figure}[h]
	\begin{minipage}{0.5\textwidth}
		\centering
		\begin{tikzpicture}[scale=1.2]
		
		\draw[thick](0,0) circle [x radius={sqrt(3)}, y radius={sqrt(2)}];
		
		\draw[thick,dashed,gray](0,0) circle [x radius={sqrt(3-2.332)}, y radius={sqrt(2+2.332)}];
		
		\draw[gray] ({sqrt(5)+0.5},-0.5) -- (-0.5,{sqrt(5)+0.5});
		\draw[gray] ({-sqrt(5)-0.5},0.5) -- (0.5,{-sqrt(5)-0.5});
		\draw[gray] (0.5,{sqrt(5)+0.5}) -- ({-sqrt(5)-0.5},-0.5);
		\draw[gray] (-0.5,{-sqrt(5)-0.5}) -- ({sqrt(5)+0.5},0.5);
		
		\draw[very thick, blue] (.9, -1.208) -- (.7499, 1.27) -- (1.709, -.229) -- (.9, -1.208);
		\end{tikzpicture}
	\end{minipage}
	\begin{minipage}{0.5\textwidth}
		\centering
		\begin{tikzpicture}[scale=0.8]
		
		\draw[thick](0,0) circle [x radius={sqrt(7)}, y radius={sqrt(5)}];
		
		\draw[thick,dashed,gray](0,0) circle [x radius={sqrt(7+4.589)}, y radius={sqrt(5-4.589)}];
		
		\draw[gray] ({sqrt(12)+0.5},-0.5) -- (-0.5,{sqrt(12)+0.5});
		\draw[gray] ({-sqrt(12)-0.5},0.5) -- (0.5,{-sqrt(12)-0.5});
		\draw[gray] (0.5,{sqrt(12)+0.5}) -- ({-sqrt(12)-0.5},-0.5);
		\draw[gray] (-0.5,{-sqrt(12)-0.5}) -- ({sqrt(12)+0.5},0.5);
		
		\draw[very thick, blue] (1, -2.07) -- (2.433, -.880) -- (-2.540, -.4935) -- (1, -2.07);
		\end{tikzpicture}
	\end{minipage}
	\caption{A $3$-periodic trajectory with an ellipse along the F as caustic ($a=3$, $b=2$, $\gamma\approx2.332$) is shown on the left, while another trajectory with an ellipse along the $\mathsf{x}$-axis as caustic ($a=7$, $b=5$, $\gamma\approx -4.589$) is on the right.}
	\label{fig:3-periodic}
\end{figure}

\subsubsection*{4-periodic trajectories}
There is a $4$-periodic trajectory of the billiard within \refeq{eq:ellipse}, with a non-degenerate caustic $\C_{\gamma}$ in the Minkowski plane if and only if $B_{3}=0$.
We solve the equation
\begin{equation}\label{eqn:4-periodic}
B_{3}=-\dfrac{(ab+a\gamma+b\gamma)(ab+a\gamma-b\gamma)(ab-a\gamma-b\gamma)}{16(ab\gamma)^{5/2}}=0,
\end{equation}
which yields the following solutions for the parameter $\gamma$ for the caustic
\begin{equation}\label{4pCaustic}
{\gamma}_{1}=-\dfrac{ab}{a+b},\qquad {\gamma}_{2}=-\dfrac{ab}{a-b}, \qquad {\gamma}_{3}=\dfrac{ab}{a+b}.
\end{equation}
Since $\gamma_{1}\in(-b,0)$, ${\gamma}_{3} \in (0,a)$ and ${\gamma}_{2} \notin (-b,a)$, therefore conic $\C_{{\gamma}_{2}}$ is a hyperbola whereas conics $\C_{{\gamma}_{1}}$ and $\C_{{\gamma}_{3}}$ are ellipses.

In Figures \ref{fig:4-periodic} and \ref{fig:4-periodich}, examples of a $4$-periodic trajectories with each type of caustic are shown.
\begin{figure}[h]
	\begin{minipage}{0.5\textwidth}
		\centering
	\begin{tikzpicture}[scale=1]

\draw[thick](0,0) circle [x radius={sqrt(2)}, y radius={sqrt(4)}];

\draw[thick,dashed,gray](0,0) circle [x radius={sqrt(2-1.333)}, y radius={sqrt(4+1.333)}];

\draw[gray] ({sqrt(6)+0.5},-0.5) -- (-0.5,{sqrt(6)+0.5});
\draw[gray] ({-sqrt(6)-0.5},0.5) -- (0.5,{-sqrt(6)-0.5});
\draw[gray] (0.5,{sqrt(6)+0.5}) -- ({-sqrt(6)-0.5},-0.5);
\draw[gray] (-0.5,{-sqrt(6)-0.5}) -- ({sqrt(6)+0.5},0.5);

\draw[very thick, blue] (1.1, -1.257) -- (.6064, 1.807) -- (1.099, 1.259) -- (.6069, -1.807) -- (1.1, -1.257);
\end{tikzpicture}
\end{minipage}
\begin{minipage}{0.5\textwidth}
	\centering
	\begin{tikzpicture}[scale=0.7]

\draw[thick](0,0) circle [x radius={sqrt(9)}, y radius={sqrt(3)}];

\draw[thick,dashed,gray](0,0) circle [x radius={sqrt(9+2.250)}, y radius={sqrt(3-2.250)}];

\draw[gray] ({sqrt(12)+0.5},-0.5) -- (-0.5,{sqrt(12)+0.5});
\draw[gray] ({-sqrt(12)-0.5},0.5) -- (0.5,{-sqrt(12)-0.5});
\draw[gray] (0.5,{sqrt(12)+0.5}) -- ({-sqrt(12)-0.5},-0.5);
\draw[gray] (-0.5,{-sqrt(12)-0.5}) -- ({sqrt(12)+0.5},0.5);

\draw[very thick, blue] (.5, -1.708) -- (2.902, -.4391) -- (-.4958, -1.708) -- (-2.902, -.4385) -- (.5, -1.708);
\end{tikzpicture}
\end{minipage}
\caption{
	A $4$-periodic trajectory with an ellipse along the $\mathsf{y}$-axis as caustic ($a=2$, $b=4$, $\gamma=4/3$) is shown on the left, while another trajectory with an ellipse along the $\mathsf{x}$-axis as caustic ($a=9$, $b=3$, $\gamma=-9/4$) is on the right.}
	\label{fig:4-periodic}
\end{figure}
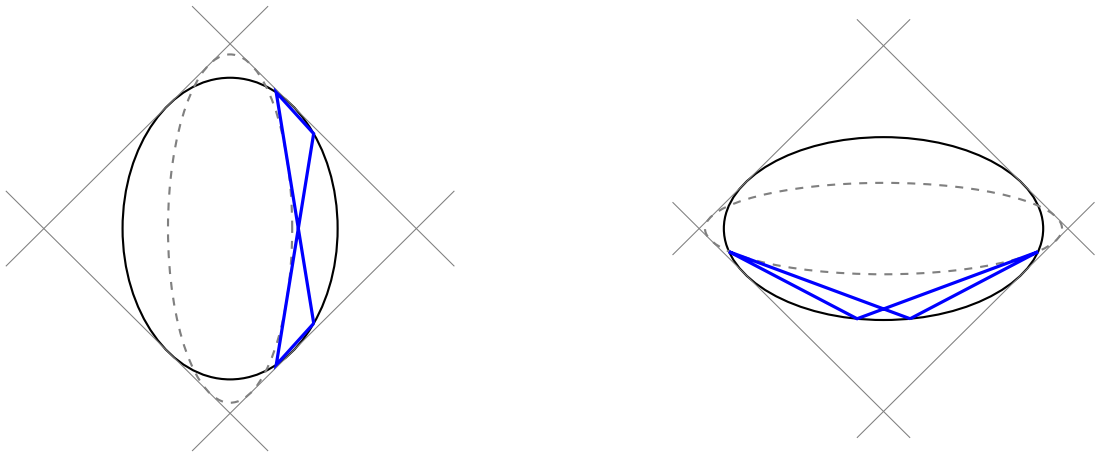

\begin{figure}[h]
	\centering
	\begin{tikzpicture}[scale=0.7]
\clip (-7,-4) rectangle (7,4);

\draw[thick](0,0) circle [x radius={sqrt(5)}, y radius={sqrt(3)}];

hyperbola lambda=-7.500
\draw[domain=-2:2,smooth,thick,dashed,variable=\t] plot ({sqrt(5+7.500)*sqrt(1-\t*\t/(3-7.500))},{\t});
\draw[domain=-2:2,smooth,thick,dashed,variable=\t] plot ({-sqrt(5+7.500)*sqrt(1-\t*\t/(3-7.500))},{\t});

\draw[gray] ({sqrt(8)+2},-2) -- (-2,{sqrt(8)+2});
\draw[gray] ({-sqrt(8)-2},2) -- (2,{-sqrt(8)-2});
\draw[gray] (2,{sqrt(8)+2}) -- ({-sqrt(8)-2},-2);
\draw[gray] (-2,{-sqrt(8)-2}) -- ({sqrt(8)+2},2);

\draw[very thick, blue] (.9, -1.586) -- (2.162, -.4427) -- (-.9015, 1.585) -- (-2.162, .4430) -- (.9, -1.586);
\end{tikzpicture}
\caption{A 4-periodic trajectory with a hyperbola along the $\mathsf{x}$-axis as caustic ($a=5$, $b=3$, $\gamma=-15/2$.}
	\label{fig:4-periodich}
\end{figure}
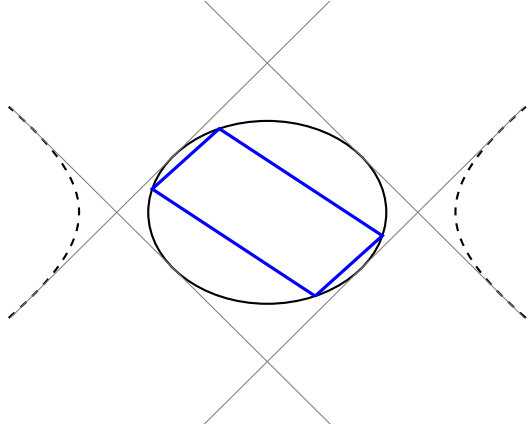

\subsubsection*{5-periodic trajectories}
There is a $5$-periodic trajectory of the billiard within \refeq{eq:ellipse}, with a non-degenerate caustic $\C_{\gamma}$ if and only if, according to Theorem \ref{th:cayley-billiard}, the caustic is an ellipse, i.e.\ $\gamma \in (-b,a)$, and $C_{2}C_{4}-C^{2}_{3}=0$, which is equivalent to:
\begin{align}
\begin{split}\label{eqn:5-periodic}
0=&  \left( a+b \right) ^{6}\gamma^{6}-2ab \left(a-b \right)  \left( a-3b \right)  \left( 3a-b \right)  \left( a+b \right)^{2}\gamma^{5}
	-a^{2}b^{2} \left( 29a^{2}-54ab+29b^{2} \right)  \left( a+b \right)^{2} \gamma^{4}
\\
&
	-36{a}^{3}{b}^{3} \left(a -b \right)  \left( a+b \right)^{2} \gamma^{3}
	-a^{4}b^{4} \left( 9a^{2}+34ab+9b^{2} \right) \gamma^{2}+10a^{5}b^{5} \left(a -b \right) \gamma+5a^{6}b^{6}.
	\end{split}
\end{align}


Examples of $5$-periodic billiard trajectories are shown in Figures \ref{fig:5-periodicy} and \ref{fig:5-periodicx}.
\begin{figure}[H]
\begin{minipage}{0.5\textwidth}
	\centering
		\begin{tikzpicture}[scale=1]

\draw[thick](0,0) circle [x radius={sqrt(5)}, y radius={sqrt(2)}];

\draw[thick,dashed,gray](0,0) circle [x radius={sqrt(5-4.7375)}, y radius={sqrt(2+4.7375)}];

\draw[gray] ({sqrt(7)+0.5},-0.5) -- (-0.5,{sqrt(7)+0.5});
\draw[gray] ({-sqrt(7)-0.5},0.5) -- (0.5,{-sqrt(7)-0.5});
\draw[gray] (0.5,{sqrt(7)+0.5}) -- ({-sqrt(7)-0.5},-0.5);
\draw[gray] (-0.5,{-sqrt(7)-0.5}) -- ({sqrt(7)+0.5},0.5);

\draw[very thick, blue] (1, -1.265) -- (.4548, 1.385) -- (.6076, -1.362) -- (1.548, 1.021) -- (2.151, .386) -- (1, -1.265);
\end{tikzpicture}
\end{minipage}
\begin{minipage}{0.5\textwidth}
	\centering
\begin{tikzpicture}[scale=1]

\draw[thick](0,0) circle [x radius={sqrt(6)}, y radius={sqrt(4)}];

\draw[thick,dashed,gray](0,0) circle [x radius={sqrt(6-1.4205)}, y radius={sqrt(4+1.4205)}];

\draw[gray] ({sqrt(10)+0.5},-0.5) -- (-0.5,{sqrt(10)+0.5});
\draw[gray] ({-sqrt(10)-0.5},0.5) -- (0.5,{-sqrt(10)-0.5});
\draw[gray] (0.5,{sqrt(10)+0.5}) -- ({-sqrt(10)-0.5},-0.5);
\draw[gray] (-0.5,{-sqrt(10)-0.5}) -- ({sqrt(10)+0.5},0.5);

\draw[very thick, blue] (2.15, -.9585) -- (2.130, 1.0) -- (1.745, 1.403) -- (2.447, -0.97e-1) -- (1.701, -1.439) -- (2.15, -.9585);

\end{tikzpicture}
\end{minipage}
\caption{$5$-periodic trajectories with an ellipse along the $\mathsf{y}$-axis as caustic.
	On the left, the particle is bouncing $4$ times off the relativistic ellipse and once off relativistic hyperbola ($a=5$, $b=2$, $\gamma \approx 4.7375$), while on the right the billiard particle is reflected twice off relativistic ellipse and $3$ times off relativistic hyperbola ($a=6$, $b=4$, $\gamma\approx1.4205$). }
	\label{fig:5-periodicy}
\end{figure}
\begin{figure}[h]
	\begin{minipage}{0.5\textwidth}
		\centering
		\begin{tikzpicture}[scale=1]

\draw[thick](0,0) circle [x radius={sqrt(6)}, y radius={sqrt(4)}];

\draw[thick,dashed,gray](0,0) circle [x radius={sqrt(6+3.9947)}, y radius={sqrt(4-3.9947)}];

\draw[gray] ({sqrt(10)+0.5},-0.5) -- (-0.5,{sqrt(10)+0.5});
\draw[gray] ({-sqrt(10)-0.5},0.5) -- (0.5,{-sqrt(10)-0.5});
\draw[gray] (0.5,{sqrt(10)+0.5}) -- ({-sqrt(10)-0.5},-0.5);
\draw[gray] (-0.5,{-sqrt(10)-0.5}) -- ({sqrt(10)+0.5},0.5);

\draw[very thick, blue] (2.3, -.6879) -- (.7881, -1.894) -- (-2.409, -.3624) -- (2.448, -0.0571) -- (-2.447, -0.09091) -- (2.3, -.6879);

\end{tikzpicture}
	\end{minipage}
	\begin{minipage}{0.5\textwidth}
		\centering
		\begin{tikzpicture}[scale=1]

\draw[thick](0,0) circle [x radius={sqrt(6)}, y radius={sqrt(4)}];

\draw[thick,dashed,gray](0,0) circle [x radius={sqrt(6+1.5413)}, y radius={sqrt(4-1.5413)}];

\draw[gray] ({sqrt(10)+0.5},-0.5) -- (-0.5,{sqrt(10)+0.5});
\draw[gray] ({-sqrt(10)-0.5},0.5) -- (0.5,{-sqrt(10)-0.5});
\draw[gray] (0.5,{sqrt(10)+0.5}) -- ({-sqrt(10)-0.5},-0.5);
\draw[gray] (-0.5,{-sqrt(10)-0.5}) -- ({sqrt(10)+0.5},0.5);

\draw[very thick, blue] (1.5, -1.581) -- (2.080, -1.056) -- (0.04168, -2.000) -- (-2.071, -1.068) -- (-1.545, -1.552) -- (1.5, -1.581);

\end{tikzpicture}
	\end{minipage}
	\caption{$5$-periodic trajectories with an ellipse along the $\mathsf{x}$-axis as caustic.
		On the left, the particle is bouncing once off the relativistic ellipse and $4$ times off relativistic hyperbola ($a=6$, $b=4$, $\gamma \approx-3.9947$), while on the right the billiard particle is reflected twice off relativistic hyperbola and $3$ times off relativistic ellipse ($a=6$, $b=4$, $\gamma\approx-1.5413$). }
	\label{fig:5-periodicx}
\end{figure}
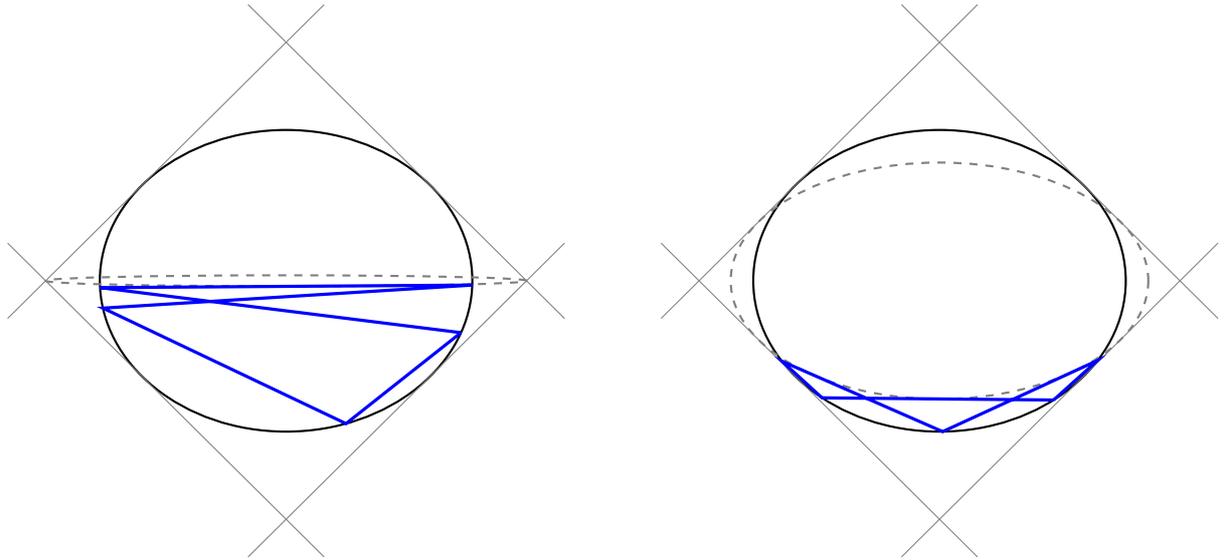

\subsubsection*{6-periodic trajectories}

There is a $6$-periodic trajectory of the billiard within \refeq{eq:ellipse}, with a non-degenerate caustic $\C_{\gamma}$ if and only if $B_{3}B_{5}-B^{2}_{4}=0$, which is equivalent to:
\begin{align}
\begin{split}\label{eqn:6-periodic}
0=&\Big({-}(a+b)^{2}\gamma^{2}+2ab ( a-b) \gamma+3a^{2}b^{2}\Big)
\Big( ( a+b)( a-3b) \gamma^{2}+2ab ( a+b) \gamma+a^{2}b^{2}\Big)
\\
&\times
\Big(  ( a+b)^{2}\gamma^{2}+2ab ( a-b) \gamma + a^{2}b^{2}\Big)
\Big ( -( a+b)( 3a-b) \gamma^{2}-2ab( a+b) \gamma+a^{2}b^{2}\Big)
.
\end{split}
\end{align}
The first factor, $-(a+b)^{2}\gamma^{2}+2ab ( a-b)\gamma+3a^2b^2$, is a constant multiple of $C_2$ (see Equation \refeq{eqn:3-periodic}), thus
it produces $3$-periodic trajectories, which have already been studied.

The discriminant of the third factor $( a+b)^{2}\gamma^{2}+2ab ( a-b) \gamma + a^{2}b^{2}$ is $-16a^3b^3$, which is negative, therefore the expression has no real roots in $\gamma$.

Next, we consider the second factor:
$
( a+b)( a-3b) \gamma^{2}+2ab ( a+b) \gamma+a^{2}b^{2}=0,
$
which has two real solutions:
\begin{equation*}
\gamma= \frac{ab}{(a+b)(a-3b)}\Big(-a-b\pm 2\sqrt{ab+b^{2}}\Big)
\end{equation*}

Finally we consider the fourth factor:
$
-( a+b)( 3a-b) \gamma^{2}-2ab( a+b) \gamma+a^{2}b^{2}=0,
$
which yields two real solutions:
\begin{equation*}
\gamma= \frac{ab}{(a+b)(3a-b)}\Big(-a-b\pm 2\sqrt{ab+a^{2}}\Big).
\end{equation*}

Examples of $6$-periodic trajectories with hyperbolas as caustics are shown in Figure \ref{fig:6-periodich}.
\begin{figure}[H]
	\begin{minipage}{0.5\textwidth}
		\centering
		\begin{tikzpicture}[scale=.8]
		
		\draw[thick](0,0) circle [x radius={sqrt(5)}, y radius={sqrt(3)}];
		
		hyperbola lambda=-3.2264
		\draw[domain=-.8:.8,smooth,thick,dashed,variable=\t] plot ({sqrt(5+3.2264)*sqrt(1-\t*\t/(3-3.2264))},{\t});
		\draw[domain=-.8:.8,smooth,thick,dashed,variable=\t] plot ({-sqrt(5+3.2264)*sqrt(1-\t*\t/(3-3.2264))},{\t});
		
		\draw[gray] ({sqrt(8)+2},-2) -- (-2,{sqrt(8)+2});
		\draw[gray] ({-sqrt(8)-2},2) -- (2,{-sqrt(8)-2});
		\draw[gray] (2,{sqrt(8)+2}) -- ({-sqrt(8)-2},-2);
		\draw[gray] (-2,{-sqrt(8)-2}) -- ({sqrt(8)+2},2);
		
		\draw[very thick, blue] (2, -0.7746) -- (1.366, -1.372) -- (-2.232, -.1051) -- (2.001, .7732) --(1.364, 1.372)--(-2.232, .1055)-- (2, -0.7746);
		\end{tikzpicture}
	\end{minipage}
	\begin{minipage}{0.5\textwidth}
		\centering
		\begin{tikzpicture}[scale=.8]
		
		\draw[thick](0,0) circle [x radius={sqrt(3)}, y radius={sqrt(7)}];
		
		hyperbola lambda=3.1189
		\draw[domain=-.5:.5,smooth,thick,dashed,variable=\t] plot ({\t},{sqrt(7+3.1189)*sqrt(1-\t*\t/(3-3.1189)});
		\draw[domain=-.5:.5,smooth,thick,dashed,variable=\t] plot ({\t},{-sqrt(7+3.1189)*sqrt(1-\t*\t/(3-3.1189)});
		
		\draw[gray] ({sqrt(10)+2},-2) -- (-2,{sqrt(10)+2});
		\draw[gray] ({-sqrt(10)-2},2) -- (2,{-sqrt(10)-2});
		\draw[gray] (2,{sqrt(10)+2}) -- ({-sqrt(10)-2},-2);
		\draw[gray] (-2,{-sqrt(10)-2}) -- ({sqrt(10)+2},2);
		
		\draw[very thick, blue] (.4, -2.574) -- (-.1910, 2.630) -- (-1.653, -.7911) -- (-.3996, -2.574)--(.1911, 2.629)--(1.652, -.7943)--(.4, -2.574);
		
		\end{tikzpicture}
	\end{minipage}
	\caption{A 6-periodic trajectory with a hyperbola along the $\mathsf{x}$-axis as caustic ($a=5$, $b=3$, $\gamma\approx -3.2264$ is shown on the left, while another trajectory with a hyperbola along the $\mathsf{y}$-axis  as caustic ($a=3$, $b=7$ and $\gamma\approx3.1189$) is on the right. On the left, the particle bounces off the relativistic ellipse twice and 4 times the relativistic hyperbola while on the right the particle bounces off the relativistic ellipse 4 times and the relativistic hyperbola twice.}
	\label{fig:6-periodich}
\end{figure}

\subsubsection*{7-periodic trajectories}

According to Theorem \ref{th:cayley-billiard}, there is a $7$-periodic trajectory of the billiard within \refeq{eq:ellipse}, with a non-degenerate caustic $\C_{\gamma}$ if and only if the caustic is an ellipse, i.e.\ $\gamma \in (-b,a)$, and
	 \begin{equation*}
	 \left|
	 \begin{array}{ccc}
	 C_2 & C_3 & C_4
	 \\
	 C_3 & C_4 & C_5
	 \\
	 C_4 & C_5 & C_6
	 \end{array}
	 \right|=0,
	 \end{equation*}
which is equivalent to:
\begin{align}
\begin{split}\label{eqn:7-periodic}
0
=&
 - \left( a+b \right) ^{12}\gamma^{12}+4ab \left( a-b \right)  \left( a-3b \right)  \left( 3a-b \right)  \left( {a}^{2}-6ab+{b}^{2} \right)\left( a+b \right) ^{6}{\gamma}^{11}
  \\
&
+2{a}^{2}{b}^{2} \left( 59{a}^{4}-332{a}^{3}b+626{a}^{2}{b}^{2}-332a{b}^{3}+59{b}^{4} \right)  \left( a+b \right) ^{6}{\gamma}^{10}\\
&+28{a}^{3}{b}^{3} \left( a-b \right)  \left( 13{a}^{2}-38ab+13{b}^{2} \right)  \left( a+b \right)^{6}{\gamma}^{9}\\
&
+{a}^{4}{b}^{4} \left( 7{a}^{2}+30ab+7{b}^{2} \right)  \left( 63{a}^{4}-84{a}^{3}b-38{a}^{2}{b}^{2}-84a{b}^{3}+63{b}^{4} \right)\left( a+b \right) ^{2}{\gamma}^{8}
\\
&
 -8{a}^{5}{b}^{5} \left( a-b \right)  \left( 21{a}^{4}-420{a}^{3}b-50{a}^{2}{b}^{2}-420a{b}^{3}+21{b}^{4} \right)  \left( a+b \right) ^{2}{\gamma}^{7}\\
 &-12{a}^{6}{b}^{6} \left( 105{a}^{4}-420{a}^{3}b+422{a}^{2}{b}^{2}-420a{b}^{3}+105{b}^{4} \right)  \left( a+b \right) ^{2}{\gamma}^{6}\\
 &
-24{a}^{7}{b}^{7} \left( a-b \right)  \left( 75{a}^{2}-106ab+75{b}^{2} \right)  \left( a+b \right) ^{2}{\gamma}^{5}
\\
&
-3{a}^{8}{b}^{8} \left( 437{a}^{2}-726ab+437{b}^{2} \right)  \left( a+b \right)^{2}{\gamma}^{4}
-4{a}^{9}{b}^{9} \left( a-b \right)  \left( 121{a}^{2}+250ab+121{b}^{2} \right) {\gamma}^{3}
\\&
-14{a}^{10}{b}^{10} \left( 3{a}^{2}+14ab+3{b}^{2} \right) {\gamma}^{2}
+28{a}^{11}{b}^{11} \left( a-b \right) \gamma+7{a}^{12}{b}^{12}.
\end{split}
\end{align}

Examples of $7$-periodic trajectories are shown in Figure \ref{fig:7-periodic}.

\begin{figure}[H]
	\begin{minipage}{0.5\textwidth}
		\centering
		\begin{tikzpicture}[scale=1]
		
	\draw[thick](0,0) circle [x radius={sqrt(3)}, y radius={sqrt(7)}];
	
	\draw[thick,dashed,gray](0,0) circle [x radius={sqrt(3+6.9712)}, y radius={sqrt(7-6.9712)}];
	
	\draw[gray] ({sqrt(10)+0.5},-0.5) -- (-0.5,{sqrt(10)+0.5});
	\draw[gray] ({-sqrt(10)-0.5},0.5) -- (0.5,{-sqrt(10)-0.5});
	\draw[gray] (0.5,{sqrt(10)+0.5}) -- ({-sqrt(10)-0.5},-0.5);
	\draw[gray] (-0.5,{-sqrt(10)-0.5}) -- ({sqrt(10)+0.5},0.5);
	
	\draw[very thick, blue] (0, -2.64575) -- (1.45903, -1.42579) -- (-1.70584, -.458533) -- (1.72848, -.169774) -- (-1.72849, -.169638)--(1.70591, -.457905) --(-1.45976, -1.42403)-- (0, -2.64575);
	
		\end{tikzpicture}
	\end{minipage}
	\begin{minipage}{0.5\textwidth}
		\centering
		\begin{tikzpicture}[scale=1]
		
		\draw[thick](0,0) circle [x radius={sqrt(7)}, y radius={sqrt(3)}];
	
	\draw[thick,dashed,gray](0,0) circle [x radius={sqrt(7-6.9712)}, y radius={sqrt(3+6.9712)}];
	
	\draw[gray] ({sqrt(10)+0.5},-0.5) -- (-0.5,{sqrt(10)+0.5});
	\draw[gray] ({-sqrt(10)-0.5},0.5) -- (0.5,{-sqrt(10)-0.5});
	\draw[gray] (0.5,{sqrt(10)+0.5}) -- ({-sqrt(10)-0.5},-0.5);
	\draw[gray] (-0.5,{-sqrt(10)-0.5}) -- ({sqrt(10)+0.5},0.5);
	
	\draw[very thick, blue] (1.5, -1.427) -- (.4857, 1.703) -- (.1761, -1.728) -- (.1651, 1.73) -- (.4356, -1.708)--(1.360, 1.484)--(2.641, -0.099)-- (1.5, -1.427);
		\end{tikzpicture}
	\end{minipage}
	\caption{A 7-periodic trajectory with an ellipse along the $\mathsf{x}$-axis as caustic ($a=3$, $b=7$, $\gamma \approx -6.9712$) is shown on the left, while another trajectory with an ellipse along the $\mathsf{y}$-axis  as caustic ($a=7$, $b=3$ and $\gamma \approx6.9712$) is on the right. On the left, the particle bounces once off the relativistic ellipse  and 6 times off the relativistic hyperbola while on the right the particle bounces 6 times off the relativistic ellipse and once off the relativistic hyperbola.}
	\label{fig:7-periodic}
\end{figure}

\subsubsection*{8-periodic trajectories}
There is an $8$-periodic trajectory of the billiard within ellipse \refeq{eq:ellipse}, with a non-degenerate caustic $\C_{\gamma}$ if and only if
\begin{equation*}
\left|
\begin{array}{ccc}
B_3 & B_4 & B_5
\\
B_4 & B_5 & B_6
\\
B_5 & B_6 & B_7
\end{array}
\right|=0,
\end{equation*}
which is equivalent to:
\begin{align}
\begin{split}\label{eqn:8-periodic}
0=
& \left( ab-a\gamma-b\gamma \right)  \left( ab+a\gamma+b\gamma \right)  \left( ab+a\gamma-b\gamma \right)
\\
&
 \Big(  \left( a+b \right)^{4}{\gamma}^{4}-4ab \left( a+b \right)  \left( -b+a \right) ^{2}{\gamma}^{3}-2{a}^{2}{b}^{2} \left( a+b \right)  \left( 5a-3b \right) {\gamma}^{2}-4{a}^{3}{b}^{3} \left( a+b \right) \gamma\\
&
+{a}^{4}{b}^{4} \Big)  \Big(  \left( a+b \right) ^{4}{\gamma}^{4}+4ab \left( a+b \right)  \left( -b+a \right) ^{2}{\gamma}^{3}+2{a}^{2}{b}^{2} \left( a+b \right)  \left( 3a-5b \right) {\gamma}^{2}\\
&
+4{a}^{3}{b}^{3} \left( a+b \right) \gamma
+{a}^{4}{b}^{4} \Big)
\Big(  \left( {a}^{2}-6ab+{b}^{2} \right)  \left( a+b \right) ^{2}{\lambda}^{4}+4ab \left( -b+a \right)  \left( a+b \right) ^{2}{\gamma}^{3}\\
&
+2{a}^{2}{b}^{2} \left( 3{a}^{2}+2ab+3{b}^{2} \right) {\gamma}^{2}
+4{a}^{3}{b}^{3} \left( -b+a \right) \gamma+{a}^{4}{b}^{4} \Big).
\end{split}
\end{align}

In Figures \ref{fig:8-periodich} and \ref{fig:8-periodice}, three examples of $8$-periodic trajectories are shown.
\begin{figure}[H]
\begin{minipage}{0.5\textwidth}
	\centering
	\begin{tikzpicture}[scale=0.7]

\draw[thick](0,0) circle [x radius={sqrt(6)}, y radius={sqrt(3)}];

hyperbola lambda=-3.0151,
\draw[domain=-.2:.2,smooth,thick,dashed,variable=\t] plot ({sqrt(6+3.0151)*sqrt(1-\t*\t/(3-3.0151))},{\t});
\draw[domain=-.2:.2,smooth,thick,dashed,variable=\t] plot ({-sqrt(6+3.0151)*sqrt(1-\t*\t/(3-3.0151))},{\t});

\draw[gray] ({sqrt(9)+2},-2) -- (-2,{sqrt(9)+2});
\draw[gray] ({-sqrt(9)-2},2) -- (2,{-sqrt(9)-2});
\draw[gray] (2,{sqrt(9)+2}) -- ({-sqrt(9)-2},-2);
\draw[gray] (-2,{-sqrt(9)-2}) -- ({sqrt(9)+2},2);

\draw[very thick, blue] (1.5, -1.369) -- (2.250, -.685) -- (-2.448, -0.0550) -- (2.439, .1630)--(-1.510, 1.364)--(-2.247, .689)--(2.448, 0.0557)--(-2.439, -.1615) -- (1.5, -1.369);
\end{tikzpicture}
\end{minipage}
\begin{minipage}{0.5\textwidth}
	\centering
	\begin{tikzpicture}[scale=0.7]

\draw[thick](0,0) circle [x radius={sqrt(6)}, y radius={sqrt(3)}];

hyperbola lambda=6.9168
\draw[domain=-1.2:1.2,smooth,thick,dashed,variable=\t] plot ({\t},{sqrt(3+6.9168)*sqrt(1-\t*\t/(6-6.9168)});
\draw[domain=-1.2:1.2,smooth,thick,dashed,variable=\t] plot ({\t},{-sqrt(3+6.9168)*sqrt(1-\t*\t/(6-6.9168)});

\draw[gray] ({sqrt(9)+2},-2) -- (-2,{sqrt(9)+2});
\draw[gray] ({-sqrt(9)-2},2) -- (2,{-sqrt(9)-2});
\draw[gray] (2,{sqrt(9)+2}) -- ({-sqrt(9)-2},-2);
\draw[gray] (-2,{-sqrt(9)-2}) -- ({sqrt(9)+2},2);

\draw[very thick, blue] (1.5, -1.369) -- (.1935, 1.727) -- (-.8948, -1.612) -- (-2.370, .437)--(-1.501, 1.369)--(-.1936, -1.727)--(.8948, 1.612)--(2.370, -.437)-- (1.5, -1.369);

\end{tikzpicture}

\end{minipage}
	\caption{On the left, an $8$-periodic trajectory with a hyperbola along $\mathsf{x}$-axis as caustic ($a=6$, $b=3$, $\gamma\approx-3.0151$), with $2$ vertices on relativistic ellipses and $6$ on relativistic hyperbolas.
On the right, an $8$-periodic trajectory with a hyperbola along $\mathsf{y}$-axis as caustic ($a=6$, $b=3$, $\gamma\approx6.9168$), with $6$ vertices on relativistic ellipses and $2$ on relativistic hyperbolas.	
}
	\label{fig:8-periodich}
\end{figure}
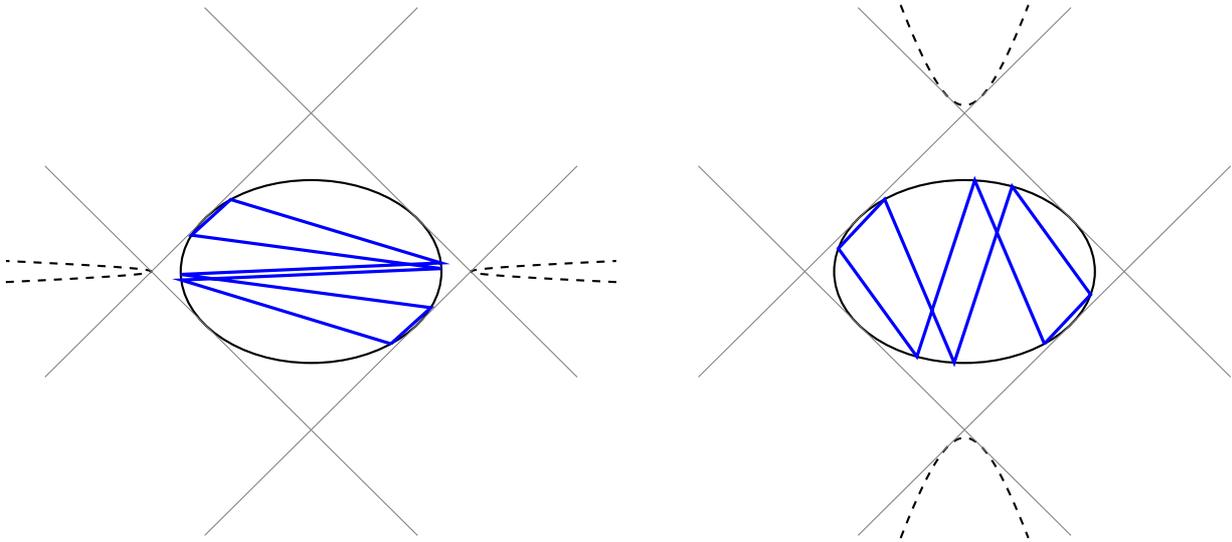
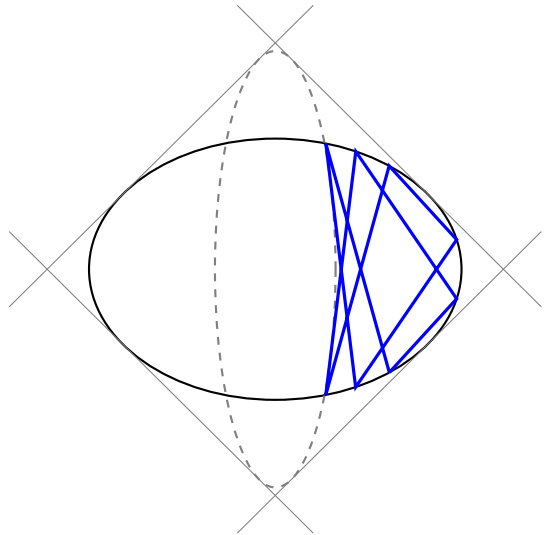
\begin{figure}[H]
	\centering
		\begin{tikzpicture}[scale=1]

\draw[thick](0,0) circle [x radius={sqrt(6)}, y radius={sqrt(3)}];

\draw[thick,dashed,gray](0,0) circle [x radius={sqrt(6-5.3707)}, y radius={sqrt(3+5.3707)}];

\draw[gray] ({sqrt(9)+0.5},-0.5) -- (-0.5,{sqrt(9)+0.5});
\draw[gray] ({-sqrt(9)-0.5},0.5) -- (0.5,{-sqrt(9)-0.5});
\draw[gray] (0.5,{sqrt(9)+0.5}) -- ({-sqrt(9)-0.5},-0.5);
\draw[gray] (-0.5,{-sqrt(9)-0.5}) -- ({sqrt(9)+0.5},0.5);

\draw[very thick, blue] (1.5, -1.369) -- (.6650, 1.667) -- (1.057, -1.564) -- (2.386, .392) -- (1.498, 1.370)--(.6646, -1.667) --(1.059, 1.561)--(2.386, -.392)-- (1.5, -1.369);

\end{tikzpicture}
	\caption{An $8$-periodic trajectory with an ellipse along $\mathsf{y}$-axis as caustic ($a=6$, $b=3$, $\gamma\approx5.3707$).
	There are $2$ reflections off relativistic hyperbola and $6$ off relativistic ellipses. }
	\label{fig:8-periodice}
\end{figure}

\subsubsection*{Summary of numbers of touching points with relativistic ellipses and hyperbolas}\label{sec:table}
In the table below, we summarise the examples given in this section.
Here, $n_1$ and $n_2$ represent the numbers of bouncing points off relativistic ellipses and relativistic hyperbolas respectively.
\begin{longtable}[h]{ |c|c|c|c| }
	\hline
	Period $n_1+n_2$ & Caustic & $n_{1}$ & $n_{2}$  \\
	\hline
	\endhead
	\hline
	\endfoot
		\hline
	$n=3$&	{Ellipse along $\mathsf{y}$-axis} & 2 & 1 \\
	 &	{Ellipse along $\mathsf{x}$-axis} & 1 & 2 \\
		\hline
			$n=4$&	{Ellipse along x-axis} & 2 & 2 \\
		&	{Ellipse along y-axis} & 2 & 2 \\
			&	{Hyperbola along x-axis} & 2 & 2 \\
		\hline
			$n=5$&	{Ellipse along y-axis} & 2 & 3 \\
		&	{Ellipse along x-axis} & 3 & 2 \\
			&	{Ellipse along y-axis} & 4 & 1 \\
		&	{Ellipse along x-axis} & 1 & 4 \\
		\hline
			$n=6$&	{Hyperbola along x-axis} & 2 & 4 \\
			& {Ellipse along y-axis} & 4 & 2 \\
		\hline
			$n=7$&	{Ellipse along x-axis} & 1 & 6 \\
		&	{Ellipse along y-axis} & 6 & 1 \\
		\hline
		$n=8$&	{Hyperbola along x-axis} & 2 & 6 \\
		&	{Hyperbola along y-axis} & 6 & 2 \\
			&	{Ellipse along x-axis} & 6 & 2 \\
		\hline
\end{longtable}

\subsection{Cayley-type conditions and discriminantly separable polynomials}

Similarly to the case of Euclidean plane \cite{DragRadn2019rcd}, the Cayley-type conditions obtained above have a very interesting algebraic structure.
Namely, the numerators of the corresponding expressions are polynomials in $3$ variables.
As examples below show, those polynomials have factorizable discriminants which, after a change of varibles, lead to discriminantly separable polynomials in the sense of the following definition.

\begin{definition}[\cite{Drag2010}]
	A polynomial $F(x_{1},\dots, x_{n})$ is \emph{discriminantly separable}
	if there exist polynomials $f_{1}(x_1),\dots , f_{n}(x_n)$ such that the discriminant $\mathcal{D}_{x_i}F$
	of F with respect to $x_{i}$ satisfies:
	
	$$
	\mathcal{D}_{x_i}F(x_{1},\dots,\hat{x}_i,\dots, x_{n})= \prod_{j\neq i}^{}f_{j}(x_{j}),
	$$
	for each $i=1,\dots,n$.
\end{definition}

Discriminantly factorizable polynomials were introduced in \cite{Drag2012} in connection with $n$-valued groups.
Various applications of discriminantly separable polynomials in continuous and discrete integrable systems were presented in \cites{DragKuk2014jgm,DragKuk2014rcd,DragKuk2017}.
The connection between Cayley-type conditions in the Euclidean setting and discriminantly factorizable and separable polynomials has been observed in \cite{DragRadn2019rcd}.
As examples below show, the Cayley conditions in the Minkowski plane provide examples of discriminantly factorisable polynomials which, after a change of variables, have separable discriminants.
It would be interesting to establish this relationship as a general statement.

\begin{example}\label{ex:G2}
The expression \refeq{eqn:3-periodic} is: 	
\begin{equation*}
\mathsf{G}_{2}(\gamma,a,b)= - \left( a+b \right) ^{2}{\gamma}^{2}+2\,ab \left( a-b \right) \gamma+3\,{a}^{2}{b}^{2},
\end{equation*}
and its discriminant with respect to $\gamma$:
$$
\mathcal{D}_{\gamma}\mathsf{G}_{2}=2^4 \left( {a}^{2}+ab+{b}^{2} \right) {a}^{2}{b}^{2},
$$
which is obviously factorizable.
\end{example}

\begin{example}
The expression \refeq{eqn:4-periodic} is:
$$
\mathsf{G}_{3}(\gamma,a,b)
=-(ab+a\gamma+b\gamma)(ab+a\gamma-b\gamma)(ab-a\gamma-b\gamma)
$$
and its discriminant with respect to $\gamma$ is factored as:
	$$
	\mathcal{D}_{\gamma}\mathsf{G}_{3}=2^6{a}^{8}{b}^{8} \left( a+b \right) ^{2}.
	$$
\end{example}

\begin{example}
The expression \refeq{eqn:5-periodic} is:
\begin{align*}
	\mathsf{G}_{6}(\gamma,a,b)=&
	  \left( a+b \right) ^{6}\gamma^{6}-2ab \left( a-b \right)  \left( a-3b \right)  \left( 3a-b \right)  \left( a+b \right)^{2}\gamma^{5}\\
&
-a^{2}b^{2} \left( 29a^{2}-54ab+29b^{2} \right)  \left( a+b \right)^{2} \gamma^{4}-36{a}^{3}{b}^{3} \left(a -b \right)  \left( a+b \right)^{2} \gamma^{3}\\
&
-a^{4}b^{4} \left( 9a^{2}+34ab+9b^{2} \right) \gamma^{2}+10a^{5}b^{5} \left(a -b \right) \gamma+5a^{6}b^{6}.
\end{align*}
is discriminantly factorizable since its discriminant with respect to $\gamma$ is:
	$$
		\mathcal{D}_{\gamma}\mathsf{G}_{6}=-5\cdot 2^{44} ( 27{a}^{6}+81{a}^{5}b+322\,{a}^{4}{b}^{2}+509{a}^{3}{b}^{3}+322{a}^{2}{b}^{4}
	+81a{b}^{5}+27{b}^{6} )  \left( a+b \right) ^{8}{b}^{38}{a}^{38}.
	$$
\end{example}

\begin{example}
Let us denote the expression \refeq{eqn:6-periodic} as:
	\begin{align*}
	\mathsf{G}_{8}(\gamma,a,b)=&( 3a-b )  ( a-3b )  ( a+b ) ^{6}{\gamma}^{8}
	+8ab ( a-b )  ( a+b ) ^{6}{\gamma}^{7}
	\\&
	-4\,{a}^{2}{b}^{2} ( 3\,{a}^{4}-24\,{a}^{3}b+10\,{a}^{2}{b}^{2}
	-24\,a{b}^{3}+3\,{b}^{4} )  \left( a+b \right) ^{2}
	{\gamma}^{6}
	\\&
	-8\,{a}^{3}{b}^{3} \left( a-b \right)  \left( 9\,{a}^{2}-14\,ab+9\,{b}^{2} \right)  \left( a+b \right) ^{2}{\gamma}^{5}\\
	&
	-10\,{a}^{4}{b}^{4} \left( 11\,{a}^{2}-18\,ab+11\,{b}^{2} \right)  \left( a+b \right) ^{2}{\gamma}^{4}-72\,{a}^{5}{b}^{5} \left( a-b \right)  \left( a+b \right) ^{2}{\gamma}^{3}\\
	&
	-4\,{a}^{6}{b}^{6} \left( a+3\,b \right)  \left( 3\,a+b \right) {\gamma}^{2}+8\,{a}^{7}{b}^{7} ( a-b ) \gamma+3\,{a}^{8}{b}^{8}
	\end{align*}
We find that the discriminant of $\mathsf{G}_{8}$ with respect to $\gamma$ factors as:
	\begin{equation*}
	\mathcal{D}_{\gamma}\mathsf{G}_{8}=-2^{88}\, \left( {a}^{2}+ab+{b}^{2} \right)  \left( a+b \right) ^{18}\\
	\mbox{}{b}^{74}{a}^{74}.
	\end{equation*}
\end{example}
	\begin{example}
		The discriminant $\mathcal{D}_{\gamma}\mathsf{G}_{12}$ of the expression in \refeq{eqn:7-periodic} is:
		\begin{align*}
	\mathcal{D}_{\gamma}\mathsf{G}_{12}=&-2^{184}\cdot7^{2}
	\left( a+b \right) ^{40}{(ab)}^{172}\times
	\\&\times
		 ( 84375{a}^{12}+506250{a}^{11}b
		+4266243\,{a}^{10}{b}^{2}+16690590\,{a}^{9}{b}^{3}+34989622\,{a}^{8}{b}^{4}
		\\&\quad
		+45383698\,{a}^{7}{b}^{5}+46564971\,{a}^{6}{b}^{6}+
		45383698\,{a}^{5}{b}^{7}+34989622\,{a}^{4}{b}^{8}+16690590\,{a}^{3}{b}^{9}
		\\&\quad
		+4266243\,{a}^{2}{b}^{10}+506250\,a{b}^{11}
		+84375\,{b}^{12} )  ,
		\end{align*}
thus $\mathsf{G}_{12}$ is a discriminantly factorizable polynomial.
	\end{example}

\begin{example}\label{ex:G15}
The discriminant $\mathcal{D}_{\gamma}\mathsf{G}_{15}$ of the expression  \refeq{eqn:8-periodic} is:
	\begin{align*}
	\mathcal{D}_{\gamma}\mathsf{G}_{15}=&-2^{246}
	(ab)^{278}\left( 27\,{a}^{2}+46\,ab+27\,{b}^{2} \right)  \left( a+b \right) ^{8}
	\times\\&\times	
	\left( {a}^{5}+5\,{a}^{4}b+10\,{a}^{3}{b}^{2}+10\,{a}^{2}{b}^{3}+5\,a{b}^{4}+{b}^{5} \right)
	\times\\&\times	
	 ( {a}^{7}+7\,{a}^{6}b+21\,{a}^{5}{b}^{2}+35\,{a}^{4}{b}^{3}+35\,{a}^{3}{b}^{4}
	+21\,{a}^{2}{b}^{5}+7\,a{b}^{6}+{b}^{7} )
	 \times\\&\times	
	 ( 8\,{a}^{26}+27\,{b}^{26}+200\,{a}^{25}b+2427\,{a}^{24}{b}^{2}+19048\,{a}^{23}{b}^{3}+108652\,{a}^{22}{b}^{4}
	+479688\,{a}^{21}{b}^{5}
	\\&\quad
	+1703702\,{a}^{20}{b}^{6}+4993208\,{a}^{19}{b}^{7}+12286692\,{a}^{18}{b}^{8}+25688608\,{a}^{17}{b}^{9}
	+46007797\,{a}^{16}{b}^{10}
	\\&\quad
	+70961808\,{a}^{15}{b}^{11}+94556312\,{a}^{14}{b}^{12}+108998288\,{a}^{13}{b}^{13}
	+108671412\,{a}^{12}{b}^{14}\\
	&\quad
	+93545968\,{a}^{11}{b}^{15}+69297712\,{a}^{10}{b}^{16}+43955208\,{a}^{9}{b}^{17}
	+23703317\,{a}^{8}{b}^{18}+10761608\,{a}^{7}{b}^{19}\\
	&\quad
	+4059132\,{a}^{6}{b}^{20}+1248808\,{a}^{5}{b}^{21}
	+305302\,{a}^{4}{b}^{22}+57048\,{a}^{3}{b}^{23}+7652\,{a}^{2}{b}^{24}+656\,a{b}^{25}
	 )
	\\&\times
	 ( 27\,{a}^{26}+8\,{b}^{26}+656\,{a}^{25}b+7652\,{a}^{24}{b}^{2}+57048\,{a}^{23}{b}^{3}+305302\,{a}^{22}{b}^{4}
	 +1248808\,{a}^{21}{b}^{5}
	 \\&\quad
	+4059132\,{a}^{20}{b}^{6}+10761608\,{a}^{19}{b}^{7}+23703317\,{a}^{18}{b}^{8}+43955208\,{a}^{17}{b}^{9}
	+69297712\,{a}^{16}{b}^{10}
	\\&\quad
	+93545968\,{a}^{15}{b}^{11}+108671412\,{a}^{14}{b}^{12}+108998288\,{a}^{13}{b}^{13}
	+94556312\,{a}^{12}{b}^{14}\\
	&\quad
	+70961808\,{a}^{11}{b}^{15}+46007797\,{a}^{10}{b}^{16}+25688608\,{a}^{9}{b}^{17}
	+12286692\,{a}^{8}{b}^{18}+4993208\,{a}^{7}{b}^{19}\\
	&\quad
	+1703702\,{a}^{6}{b}^{20}+479688\,{a}^{5}{b}^{21}+108652\,{a}^{4}{b}^{22}
	+19048\,{a}^{3}{b}^{23}+2427\,{a}^{2}{b}^{24}+200\,a{b}^{25}
	 )   ,
	\end{align*}
so $\mathsf{G}_{15}$ is a discriminantly factorizable polynomial.	
\end{example}

\begin{remark}
Since the determinants obtained in Theorem \ref{th:cayley-billiard} are symmetric in $a$, $-b$, and $\gamma$, the discriminants with respect to $a$ and $b$ of the polynomials in Examples \ref{ex:G2}--\ref{ex:G15} will be also factorizable.
\end{remark}

\begin{remark}
	We observed in the Examples \ref{ex:G2}--\ref{ex:G15} that all polynomials are discriminantly factorizable. However, it is important to note that their factors are homogeneous, thus, by a change of variables $(a,b)\mapsto (a,\hat{b})$, with $\hat{b}=\dfrac{b}{a}$, transforms the polynomials into discriminantly separable polynomials in new variables $(a, \hat{b})$:
\begin{gather*}
\mathcal{D}_{\gamma}\mathsf{G}_{2}=2^4\,a^{8} \hat{b}^{2}\left( 1+\hat{b}+\hat{b}^{2} \right),
\\	
\mathcal{D}_{\gamma}\mathsf{G}_{3}=2^6\,{a}^{18}{\hat{b}}^{8} \left( 1+\hat{b} \right) ^{2},
\\
	\mathcal{D}_{\gamma}\mathsf{G}_{6}=-52^{44}a^{90} \hat{b}^{38}( 27+81\hat{b}+322\hat{b}^{2}+509\hat{b}^{3}+322\hat{b}^{4}
	+81\hat{b}^{5}+27{b}^{6} )  \left( a+b \right) ^{8},
\\
	\mathcal{D}_{\gamma}\mathsf{G}_{8}=-2^{88}\, a^{168} \hat{b}^{74}\left( 1+\hat{b}+\hat{b}^{2} \right)  \left( 1+\hat{b} \right) ^{18},
\end{gather*}
\end{remark}

\section{Elliptic periodic trajectories}\label{sec:elliptic}

Points of the plane which are symmetric with respect to the coordinate axes share the same elliptic coordinates, thus there is no bijection between the elliptic and the Cartesian coordinates. Thus, we introduce a separate notion of periodicity in elliptic coordinates.

\begin{definition}
A billiard trajectory is \emph{$n$-elliptic periodic} if it is $n$-periodic in elliptic coordinates joined to the confocal family \refeq{eq:confocal.conics}.
\end{definition}

Now, we will derive algebro-geometric conditions for elliptic periodic trajectories.

\begin{theorem}\label{th:elliptic-periodic}
A billiard trajectory within $\E$ with the caustic $\C_{\gamma}$ is $n$-elliptic periodic without being $n$-periodic if and only if one of the following conditions is satisfied on $\Curve$:
	\begin{itemize}
		\item[(a)] $\C_{\gamma}$ is an ellipse,  $0<\gamma<a$, and $nQ_{0}-(n-1)Q_{\gamma}-Q_{-b}\sim0$;
		\item[(b)] $\C_{\gamma}$ is an ellipse,  $-b<\gamma<0$, and $nQ_{0}-(n-1)Q_{\gamma}-Q_{a}\sim0$;
		
		\item[(c)] $\C_{\gamma}$ is a hyperbola, $n$ is even and $nQ_{0}-(n-2)Q_{\gamma}-Q_{-b}-Q_{a}\sim0$;
		\item[(d)]  $\C_{\gamma}$ is a hyperbola, $n$ is odd, and $nQ_{0}-(n-1)Q_{\gamma}-Q_a\sim0$;
		\item[(e)]  $\C_{\gamma}$ is a hyperbola, $n$ is odd, and $nQ_{0}-(n-1)Q_{\gamma}-Q_{-b}\sim0$.
	\end{itemize}
	Moreover, such trajectories are always symmetric with respect to the origin in Case (c).
	They are symmetric with respect to the $\mathsf{x}$-axis in Cases (b) and (d), and with respect to the $\mathsf{y}$-axis in Cases (a) and (e).
\end{theorem}
\begin{proof}
	Let $M_0$ be the initial point of a given $n$-elliptic periodic trajectory, and $M_1$ the next point on the trajectory with the same elliptic coordinates.
	Then, integrating \eqref{eq:diff-eq} $M_0$ to $M_1$ along the trajectory, we get:
	$$
	n_1(Q_0-Q_{\alpha_1})+n_2(Q_{0}-Q_{\beta_1})\sim0,
	$$
	where $n=n_1+n_2$, and $n_1$ is the number of times that the particle hit the arcs of $\E$ with time-like metrics, and $n_2$ the number of times it hit the arcs with space-like metrics.
	We denoted by $\alpha_1$ the largest negative member of the set $\{a,-b,\gamma\}$, and by $\beta_1$ its smallest positive member.
	
	The trajectory is not $n$-periodic if and only if at least one of $n_1$, $n_2$ is odd, which then leads to the stated conclusions.
\end{proof}

The explicit Cayley-type conditions for elliptic periodic trajectories are:
\begin{theorem}\label{th:elliptic-cayley}
	A billiard trajectory within $\E$ with the caustic $\Q_{\gamma}$ is $n$-elliptic periodic without being $n$-periodic if and only if one of the following conditions is satisfied:
	\begin{itemize}
		\item[(a)] $\C_{\gamma}$ is an ellipse, $0<\gamma<a$, and
		\begin{gather*}
		D_1=0,
		\quad
		\left|
		\begin{array}{cc}
		D_1 & D_2
		\\
		D_2 & D_3
		\end{array}
		\right|=0,
		\quad
		\left|
		\begin{array}{ccc}
		D_1 & D_2 & D_3
		\\
		D_2 & D_3 & D_4
		\\
		D_3 & D_4 & D_5
		\end{array}
		\right|=0,
		\dots
		\quad\text{for}\quad n=2,4,6,\dots
		\\
		E_2=0,
		\quad
		\left|
		\begin{array}{cc}
		E_2 & E_3
		\\
		E_3 & E_4
		\end{array}
		\right|=0,
		\quad
		\left|
		\begin{array}{ccc}
		E_2 & E_3 & E_4
		\\
		E_3 & E_4 & E_5
		\\
		E_4 & E_5 & E_6
		\end{array}
		\right|=0,
		\dots
		\quad\text{for}\quad n=3,5,7,\dots;
		\end{gather*}
\item[(b)] $\C_{\gamma}$ is an ellipse, $-b<\gamma<0$, and
\begin{gather*}
E_1=0,
\quad
\left|
\begin{array}{cc}
E_1 & E_2
\\
E_2 & E_3
\end{array}
\right|=0,
\quad
\left|
\begin{array}{ccc}
E_1 & E_2 & E_3
\\
E_2 & E_3 & E_4
\\
E_3 & E_4 & E_5
\end{array}
\right|=0,
\dots
\quad\text{for}\quad n=2,4,6,\dots
\\
D_2=0,
\quad
\left|
\begin{array}{cc}
D_2 & D_3
\\
D_3 & D_4
\end{array}
\right|=0,
\quad
\left|
\begin{array}{ccc}
D_2 & D_3 & D_4
\\
D_3 & D_4 & D_5
\\
D_4 & D_5 & D_6
\end{array}
\right|=0,
\dots
\quad\text{for}\quad n=3,5,7,\dots;
\end{gather*}

		\item[(c)] $\Q_{\gamma}$ is a hyperbola, $n$ even and
		\begin{gather*}
		C_1=0,
		\quad
		\left|
		\begin{array}{cc}
		C_1 & C_2
		\\
		C_2 & C_3
		\end{array}
		\right|=0,
		\quad
		\left|
		\begin{array}{ccc}
		C_1 & C_2 & C_3
		\\
		C_2 & C_3 & C_4
		\\
		C_3 & C_4 & C_5
		\end{array}
		\right|=0,
		\dots
		\quad\text{for}\quad n=2,4,6,\dots
		\end{gather*}	
		\item[(d)]  $\Q_{\gamma}$ is a hyperbola, $n$ is odd, and
		$$
		D_2=0,
		\quad
		\left|
		\begin{array}{cc}
		D_2 & D_3
		\\
		D_3 & D_4
		\end{array}
		\right|=0,
		\quad
		\left|
		\begin{array}{ccc}
		D_2 & D_3 & D_4
		\\
		D_3 & D_4 & D_5
		\\
		D_4 & D_5 & D_6
		\end{array}
		\right|=0,
		\dots
		\quad\text{for}\quad n=3,5,7,\dots.
		$$	
\item[(e)] $\Q_{\gamma}$ is a hyperbola, $n$ is odd, and
$$
E_2=0,
\quad
\left|
\begin{array}{cc}
E_2 & E_3
\\
E_3 & E_4
\end{array}
\right|=0,
\quad
\left|
\begin{array}{ccc}
E_2 & E_3 & E_4
\\
E_3 & E_4 & E_5
\\
E_4 & E_5 & E_6
\end{array}
\right|=0,
\dots
\quad\text{for}\quad n=3,5,7,\dots.
$$	
	\end{itemize}
	Here, we denoted:
	\begin{gather*}
	\frac{\sqrt{\varepsilon(a-x)(b+x)(\gamma-x)}}{a-x}=D_0+D_1x+D_2x^2+\dots,
	\\
	\frac{\sqrt{\varepsilon(a-x)(b+x)(\gamma-x)}}{b+x}=E_0+E_1x+E_2x^2+\dots,
	\end{gather*}
	the Taylor expansion around $x=0$, while $B$s and $C$s are as in Theorem \ref{th:cayley-billiard}.
\end{theorem}
\begin{proof}
	(a)
	Take first $n$ even.
	Using Theorem \ref{th:elliptic-periodic}, we have:
	$$
	nQ_0
	\sim
	(n-1)Q_{\gamma}+Q_{-b}\sim (n-2)Q_{\infty}+Q_{-b}+Q_{\gamma}
	\sim
	(n-2) Q_{\infty}+Q_{\infty}+Q_{a}
	\sim
	(n-1)Q_{\infty}+Q_{a}.
	$$
	The basis of $\mathcal{L}((n-1)Q_{\infty}+Q_{a})$ is:
	$$
	1,x,x^2,\dots,x^{n/2-1},\frac{y}{x-a},\frac{xy}{x-a},
	\frac{x^{n/2-1}y}{x-a},
	$$
	thus a non-trivial linear combination of these functions with a zero of order $n$ at $x=0$ exists if and only if:
	$$
	\left|
	\begin{array}{llll}
	D_{n/2} & D_{n/2-1} & \dots & D_1\\
	D_{n/2+1} & D_{n/2} & \dots & D_2\\
	\dots\\
	D_{n-1} & D_{n-2} & \dots & D_{n/2}
	\end{array}
	\right|
	=0.
	$$
	For odd $n$, we have:
	$$
	nQ_0
	\sim
	(n-1)Q_{\gamma}+Q_{-b}
	\sim
	(n-1)Q_{\infty}+Q_{-b}.
	$$
	The basis of $\mathcal{L}((n-1)Q_{\infty}+Q_{-b})$ is:
	$$
	1,x,x^2,\dots,x^{(n-1)/2},\frac{y}{x+b},\frac{xy}{x+b}, \frac{x^{(n-1)/2-1}y}{x+b},
	$$
	thus a non-trivial linear combination of these functions with a zero of order $n$ at $x=0$ exists if and only if:
	$$
	\left|
	\begin{array}{llll}
	E_{(n-1)/2+1} & E_{(n-1)/2} & \dots & E_2\\
	E_{(n-1)/2+2} & E_{(n-1)/2+1} & \dots & E_3\\
	\dots\\
	E_{n-1} & E_{n-2} & \dots & E_{(n-1)/2+1}
	\end{array}
	\right|
	=0.
	$$

Case (b) is done similarly as (a).
	
	(c) We have
	$
	nQ_0\sim(n-2)Q_{\gamma}+Q_{-b}+Q_a\sim(n-1)Q_{\infty}+Q_{\gamma}.
	$
	
	(d) We have $nQ_0\sim (n-1)Q_{\gamma}+Q_a\sim(n-1)Q_{\infty}+Q_a$.
	
	(e) We have $nQ_0\sim (n-1)Q_{\gamma}+Q_{-b}\sim(n-1)Q_{\infty}+Q_{-b}$.
	\end{proof}

\section{Examples of elliptic periodic trajectories: $2\le n\le5$}
\label{sec:examples-elliptic}

\subsubsection*{2-elliptic periodic trajectories}
There is a $2$-elliptic periodic trajectory without being $2$-periodic of the billiard within \refeq{eq:ellipse}, with a non-degenerate caustic $\C_{\gamma}$ if and only if, according to Theorem \ref{th:elliptic-cayley} one of the following is satisfied:
\begin{itemize}
	\item the caustic is an ellipse, with $\gamma \in (0,a)$ and $D_1=0$;
	\item the caustic is an ellipse, with $\gamma \in (-b,0)$ and  $E_1=0$;
	\item the caustic is a hyperbola, $n$ is even, and $C_1=0$.
\end{itemize}
We consider the following equations:
\begin{align*}
D_1=\dfrac{(a+b)\gamma-ab}{2\sqrt{a^3b\gamma}}=0,
\quad
E_1=-\dfrac{(a+b)\gamma+ab}{2\sqrt{ab^3\gamma}}=0,
\quad
C_1= \dfrac{(a-b)\gamma+ab}{2\sqrt{ab\gamma^3}}=0,
\end{align*}
which respectively yield the solutions for the parameter $\gamma$ of the caustic:
\begin{align*}
\gamma=\dfrac{ab}{a+b},
\quad
\gamma=-\dfrac{ab}{a+b},
\quad
\gamma=-\dfrac{ab}{a-b}.
\end{align*}

Some examples of $2$-elliptic periodic trajectories without being $2$-periodic are shown in Figures \ref{fig:2-elliptic-periodice} and \ref{fig:2-elliptic-periodich}.
	\begin{figure}[h]
\begin{minipage}{0.5\textwidth}
		\centering
			\begin{tikzpicture}[scale=1]
	
	\draw[thick](0,0) circle [x radius={sqrt(5)}, y radius={sqrt(3)}];
	
	\draw[thick,dashed,gray](0,0) circle [x radius={sqrt(5+1.875)}, y radius={sqrt(3-1.875)}];
	
	\draw[gray] ({sqrt(8)+0.5},-0.5) -- (-0.5,{sqrt(8)+0.5});
	\draw[gray] ({-sqrt(8)-0.5},0.5) -- (0.5,{-sqrt(8)-0.5});
	\draw[gray] (0.5,{sqrt(8)+0.5}) -- ({-sqrt(8)-0.5},-0.5);
	\draw[gray] (-0.5,{-sqrt(8)-0.5}) -- ({sqrt(8)+0.5},0.5);
	
	\draw[very thick, blue](.3, 1.7163916) -- (-2.0697794, .65544485) -- (-.3, 1.716391);

	\end{tikzpicture}
	\end{minipage}
\begin{minipage}{0.5\textwidth}
		\centering
		\begin{tikzpicture}[scale=0.85]

\draw[thick](0,0) circle [x radius={sqrt(5)}, y radius={sqrt(7)}];

\draw[thick,dashed,gray](0,0) circle [x radius={sqrt(5-2.91667)}, y radius={sqrt(7+2.91667)}];

\draw[gray] ({sqrt(12)+0.5},-0.5) -- (-0.5,{sqrt(12)+0.5});
\draw[gray] ({-sqrt(12)-0.5},0.5) -- (0.5,{-sqrt(12)-0.5});
\draw[gray] (0.5,{sqrt(12)+0.5}) -- ({-sqrt(12)-0.5},-0.5);
\draw[gray] (-0.5,{-sqrt(12)-0.5}) -- ({sqrt(12)+0.5},0.5);

\draw[very thick, blue](2, -1.18322) -- (1.04167, 2.34112) -- (2, 1.18322);

\end{tikzpicture}
\end{minipage}
		\caption{A $2$-elliptic periodic trajectories with ellipses as caustics. On the left, the caustic is an ellipse along $\mathsf{x}$-axis ($a=5$, $b=3$, $\gamma=-15/8$), and on the right an ellipse along $\mathsf{y}$-axis
		($a=5$, $b=7$ and $\gamma =35/12$).}
\label{fig:2-elliptic-periodice}
\end{figure}
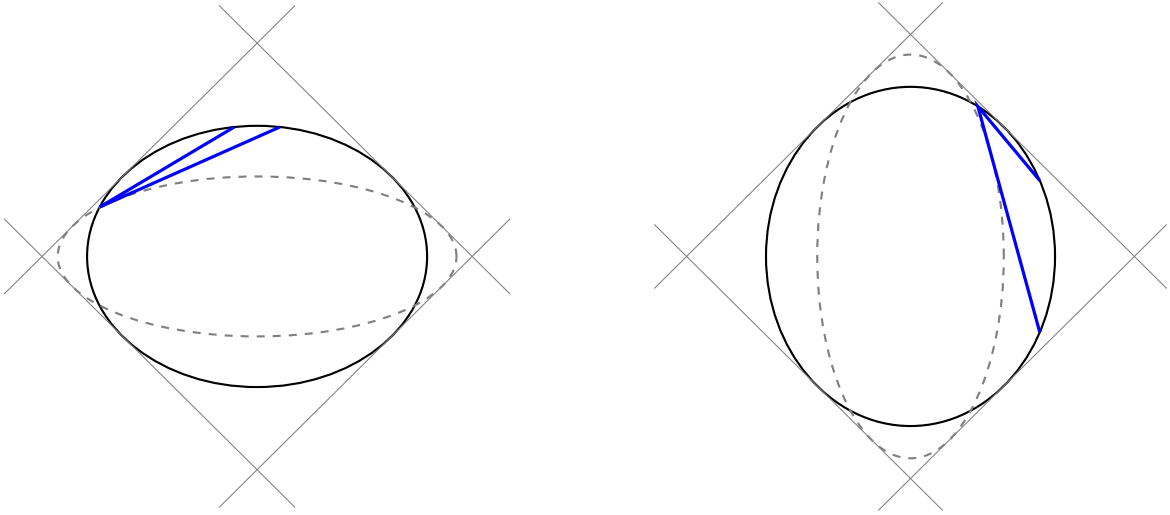
	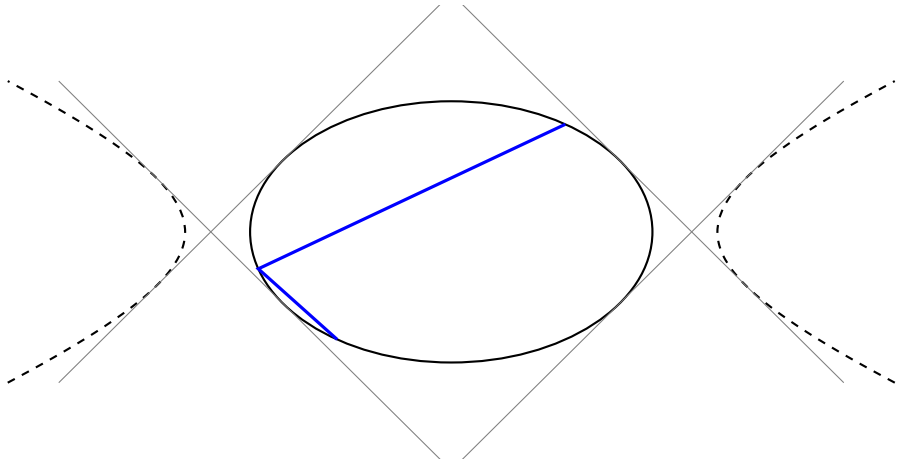
\begin{figure}[h]
		\centering
	\begin{tikzpicture}[scale=1]
\clip (-6,-3) rectangle (6,3);

\draw[thick](0,0) circle [x radius={sqrt(7)}, y radius={sqrt(3)}];

hyperbola lambda=-5.25
\draw[domain=-2:2,smooth,thick,dashed,variable=\t] plot ({sqrt(7+5.25)*sqrt(1-\t*\t/(3-5.25))},{\t});
\draw[domain=-2:2,smooth,thick,dashed,variable=\t] plot ({-sqrt(7+5.25)*sqrt(1-\t*\t/(3-5.25))},{\t});

\draw[gray] ({sqrt(10)+2},-2) -- (-2,{sqrt(10)+2});
\draw[gray] ({-sqrt(10)-2},2) -- (2,{-sqrt(10)-2});
\draw[gray] (2,{sqrt(10)+2}) -- ({-sqrt(10)-2},-2);
\draw[gray] (-2,{-sqrt(10)-2}) -- ({sqrt(10)+2},2);

\draw[very thick, blue] (1.5, 1.4267846) -- (-2.5376610, -.49001897) -- (-1.5, -1.4267846);

\end{tikzpicture}
		\caption{A $2$-elliptic periodic trajectory with a hyperbola as caustic ($a=7$, $b= 3$, $\gamma=-5.25$).}
		\label{fig:2-elliptic-periodich}
	\end{figure}

\subsubsection*{3-elliptic periodic trajectories}
There is a $3$-elliptic periodic trajectory without being $3$-periodic of the billiard within \refeq{eq:ellipse}, with a non-degenerate caustic $\C_{\gamma}$ if and only if one of the following is satisfied:
\begin{itemize}
	\item $E_2=0$ and either the caustic is an ellipse with $\gamma \in (0,a)$ or the caustic is a hyperbola with $n$ even;
	\item $D_2=0$ and either the caustic is an ellipse with $\gamma \in (-b,0)$ or the caustic is a hyperbola.
\end{itemize}
The equations $E_2=0$ and $D_2=0$ are respectively equivalent to:
\begin{gather}
-(a+b)(3a-b)\gamma^2-2ab(a+b)\gamma+a^2b^2=0,\label{eq:3-elliptic-periodic1}
\\
(a+b)(a-3b)\gamma^2+2ab(a+b)\gamma+a^2b^2 =0,\label{eq:3-elliptic-periodic2}
\end{gather}
which respectively yield the pairs of solutions for the parameter $\gamma$ of the caustic:
\begin{align*}
\gamma=\dfrac{(-a-b\pm2\sqrt{a^2+ab})ba}{(a+b)(3a-b)} ,
\quad
\gamma= \dfrac{(-a-b\pm2\sqrt{b^2+ab})ba}{(a+b)(a-3b)}.
\end{align*}

Examples of $3$-elliptic periodic trajectories which are not $3$-periodic are shown in Figures \ref{fig:3-elliptic-periodich}, \ref{fig:3-elliptic-periodicex}, \ref{fig:3-elliptic-periodicey}.

	\begin{figure}[H]
	\begin{minipage}{0.5\textwidth}
		\centering
	\begin{tikzpicture}[scale=1]
\clip (-4,-3.5) rectangle (4,3.5);

\draw[thick](0,0) circle [x radius={sqrt(6)}, y radius={sqrt(3)}];

hyperbola lambda=-3.1595918
\draw[domain=-1:1,smooth,thick,dashed,variable=\t] plot ({sqrt(6+3.1595918)*sqrt(1-\t*\t/(3-3.1595918))},{\t});
\draw[domain=-1:1,smooth,thick,dashed,variable=\t] plot ({-sqrt(6+3.1595918)*sqrt(1-\t*\t/(3-3.1595918))},{\t});

\draw[gray] ({sqrt(9)+2},-2) -- (-2,{sqrt(9)+2});
\draw[gray] ({-sqrt(9)-2},2) -- (2,{-sqrt(9)-2});
\draw[gray] (2,{sqrt(9)+2}) -- ({-sqrt(9)-2},-2);
\draw[gray] (-2,{-sqrt(9)-2}) -- ({sqrt(9)+2},2);

\draw[very thick, blue] (1,1.5811388) -- (-2.4369583, .1749772) -- (2.3389687, -.51440518) -- (1,-1.5811388);
\end{tikzpicture}
	\end{minipage}
\begin{minipage}{0.5\textwidth}
		\centering
	\begin{tikzpicture}[scale=1]
\clip (-3.5,-3.5) rectangle (3.5,3.5);


\draw[thick](0,0) circle [x radius={sqrt(3)}, y radius={sqrt(5)}];

hyperbola lambda=3.2264236
\draw[domain=-1:1,smooth,thick,dashed,variable=\t] plot ({\t},{sqrt(5+3.2264236)*sqrt(1-\t*\t/(3-3.2264236)});
\draw[domain=-1:1,smooth,thick,dashed,variable=\t] plot ({\t},{-sqrt(5+3.2264236)*sqrt(1-\t*\t/(3-3.2264236)});

\draw[gray] ({sqrt(8)+2},-2) -- (-2,{sqrt(8)+2});
\draw[gray] ({-sqrt(8)-2},2) -- (2,{-sqrt(8)-2});
\draw[gray] (2,{sqrt(8)+2}) -- ({-sqrt(8)-2},-2);
\draw[gray] (-2,{-sqrt(8)-2}) -- ({sqrt(8)+2},2);

\draw[very thick, blue] (.4, 2.1756225) -- (1.7264240, .18009374) -- (.32752004, -2.1957271) -- (-.4, 2.1756225);
\end{tikzpicture}
\end{minipage}
		\caption{$3$-elliptic periodic trajectries with hyperbolas as caustics. On the left, the caustic is orientied along the $\mathsf{x}$-axis ($a=6$, $b=3$, $\gamma\approx-3.1595918$), and on the right along $\mathsf{y}$-axis ($a=3$, $b=5$, $\gamma\approx3.2264236$).}
		\label{fig:3-elliptic-periodich}
	\end{figure}
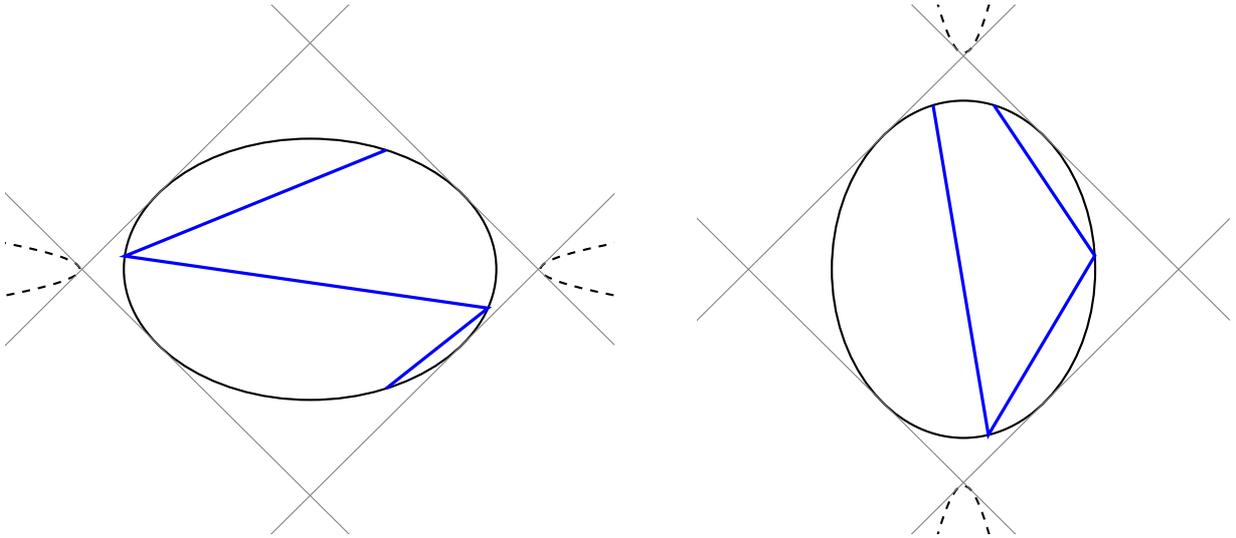
\begin{figure}[H]
	\centering
		\begin{tikzpicture}[scale=2.5]
\clip (-3.5,0.4) rectangle (3.5,1.5);

\draw[thick](0,0) circle [x radius={sqrt(9)}, y radius={sqrt(2)}];

\draw[thick,dashed,gray](0,0) circle [x radius={sqrt(9+.8831827)}, y radius={sqrt(2-.8831827)}];

\draw[gray] ({sqrt(11)+0.5},-0.5) -- (-0.5,{sqrt(11)+0.5});
\draw[gray] ({-sqrt(11)-0.5},0.5) -- (0.5,{-sqrt(11)-0.5});
\draw[gray] (0.5,{sqrt(11)+0.5}) -- ({-sqrt(11)-0.5},-0.5);
\draw[gray] (-0.5,{-sqrt(11)-0.5}) -- ({sqrt(11)+0.5},0.5);

\draw[very thick, blue](2.4, .84852814) -- (-1.3599411, 1.2605607) -- (-2.8265056, .47395874) -- (-2.4, .84852814);

\end{tikzpicture}
	\caption{A $3$-elliptic periodic trajectory with an ellipse along $\mathsf{x}$-axis as caustic ($a=9$, $b=2$, $\gamma\approx-0.8831827$).}
	\label{fig:3-elliptic-periodicex}
\end{figure}
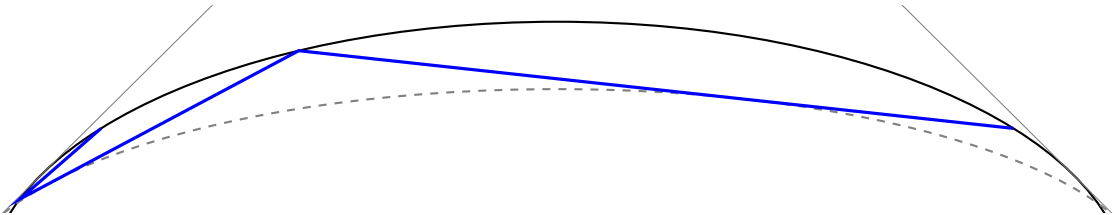
\begin{figure}[h]
	\centering
		\begin{tikzpicture}[scale=1.5]

\draw[thick](0,0) circle [x radius={sqrt(4)}, y radius={sqrt(9)}];

\draw[thick,dashed,gray](0,0) circle [x radius={sqrt(4-1.312805)}, y radius={sqrt(9+1.312805)}];

\draw[gray] ({sqrt(13)+0.5},-0.5) -- (-0.5,{sqrt(13)+0.5});
\draw[gray] ({-sqrt(13)-0.5},0.5) -- (0.5,{-sqrt(13)-0.5});
\draw[gray] (0.5,{sqrt(13)+0.5}) -- ({-sqrt(13)-0.5},-0.5);
\draw[gray] (-0.5,{-sqrt(13)-0.5}) -- ({sqrt(13)+0.5},0.5);

\draw[very thick, blue](1.4, 2.1424285) -- (.94282998, 2.6457345) -- (1.8560494, 1.1175561) -- (1.4, -2.1424285);

\end{tikzpicture}
	\caption{A $3$-elliptic periodic trajectory with an ellipse along $\mathsf{y}$-axis as caustic ($a=4$, $b=9$, $\gamma\approx1.312805$).}
	\label{fig:3-elliptic-periodicey}
\end{figure}

\subsubsection*{4-elliptic periodic trajectories}
There is a $4$-elliptic periodic trajectory without being $4$-periodic of the billiard within \refeq{eq:ellipse}, with a non-degenerate caustic $\C_{\gamma}$ if and only if, according to Theorem \ref{th:elliptic-cayley}, one of the following is satisfied:
\begin{itemize}
	\item the caustic is an ellipse, with $\gamma \in (0,a)$ and $D_3D_1-D^{2}_2=0$,
	i.e.
$$
(a+b)^4\gamma^4-4ab(a+b)(a-b)^2\gamma^3-2a^2b^2(a+b)(5a-3b)\gamma^2-4a^3b^3(a+b)\gamma+a^4b^4=0;
$$
	\item the caustic is an ellipse, with $\gamma \in (-b,0)$ and $E_3E_1-E^{2}_2=0$,
	i.e.
$$
(a+b)^4\gamma^4+4ab(a+b)(a-b)^2\gamma^3+2a^2b^2(a+b)(3a-5b)\gamma^2+4a^3b^3(a+b)\gamma+a^4b^4=0;
$$
	\item the caustic is a hyperbola and $C_3C_1-C^{2}_2=0$, i.e.
$$
 (a^2-6ab+b^2)(a+b)^2\gamma^4+4ab(a-b)(a+b)^2\gamma^3
+2a^2b^2(3a^2+2ab+3b^2)\gamma^2+4a^3b^3(a-b)\gamma+a^4b^4 =0.
$$
\end{itemize}
Each real solution $\gamma$ for the above equations for some fixed values of $a$ and $b$ will produce a $4$-elliptic periodic trajectory which is not $4$-periodic.
Some examples are shown in Figure \ref{fig:4-elliptic-periodic}.
\begin{figure}[h]
\begin{minipage}{0.5\textwidth}
	\centering
		\begin{tikzpicture}[scale=1]
\clip (-3.5,-3.5) rectangle (3.5,3.5);


\draw[thick](0,0) circle [x radius={sqrt(5)}, y radius={sqrt(3)}];

\draw[thick,dashed,gray](0,0) circle [x radius={sqrt(5-4.6212)}, y radius={sqrt(3+4.6212)}];

\draw[gray] ({sqrt(8)+2},-2) -- (-2,{sqrt(8)+2});
\draw[gray] ({-sqrt(8)-2},2) -- (2,{-sqrt(8)-2});
\draw[gray] (2,{sqrt(8)+2}) -- ({-sqrt(8)-2},-2);
\draw[gray] (-2,{-sqrt(8)-2}) -- ({sqrt(8)+2},2);

\draw[very thick, blue](2, -0.774596) -- (.718617, 1.64017) -- (.544007, -1.6800) -- (1.51403, 1.27460) -- (2, 0.774596);
\end{tikzpicture}
\end{minipage}
\begin{minipage}{0.5\textwidth}
	\centering
	\begin{tikzpicture}[scale=1]
\clip (-3.5,-3.5) rectangle (3.5,3.5);


\draw[thick](0,0) circle [x radius={sqrt(5)}, y radius={sqrt(3)}];

hyperbola lambda=-3.0243
\draw[domain=-.3:.3,smooth,thick,dashed,variable=\t] plot ({sqrt(5+3.0243)*sqrt(1-\t*\t/(3-3.0243)},{\t});
\draw[domain=-.3:.3,smooth,thick,dashed,variable=\t] plot ({-sqrt(5+3.0243)*sqrt(1-\t*\t/(3-3.0243)},{\t});

\draw[gray] ({sqrt(8)+2},-2) -- (-2,{sqrt(8)+2});
\draw[gray] ({-sqrt(8)-2},2) -- (2,{-sqrt(8)-2});
\draw[gray] (2,{sqrt(8)+2}) -- ({-sqrt(8)-2},-2);
\draw[gray] (-2,{-sqrt(8)-2}) -- ({sqrt(8)+2},2);

\draw[very thick, blue] (2, -0.774596) -- (1.35590, -1.37729) -- (-2.18, -.17) -- (2.1, 0.072622) -- (-2, 0.774596);
\end{tikzpicture}
\end{minipage}
	\caption{$4$-elliptic periodic trajectories. On the left, the caustic is an ellipse ($a=5$, $b= 3$, $\gamma\approx4.6216$), and it is a hyperbola on the right ($a=5$, $b=3$, $\gamma\approx-3.0243$).}
\label{fig:4-elliptic-periodic}
\end{figure}
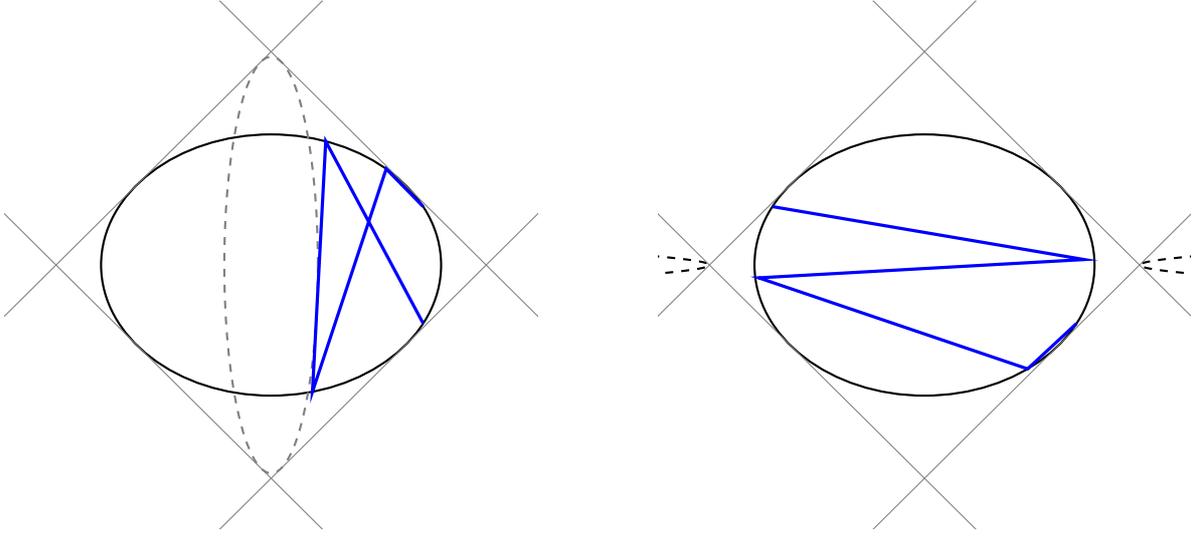

\subsubsection*{5-elliptic periodic trajectories}
According to Theorem \ref{th:elliptic-cayley}, there is a $5$-elliptic periodic trajectory without being $5$-periodic of the billiard within \refeq{eq:ellipse}, with a non-degenerate caustic $\C_{\gamma}$ if and only if one of the following is satisfied:

\begin{itemize}
	\item the caustic is an ellipse, with $\gamma \in (0,a)$ or a hyperbola and $E_2E_4-E^{2}_3=0$, i.e.
\begin{align*}
& (5a^2-10ab+b^2)(a+b)^4\gamma^6+2ab(5a-3b)(a+b)^4\gamma^5\\
&-a^2b^2(a+b)(9a^3-45a^2b-5ab^2-15b^3)\gamma^4-4a^3b^3(a+b)(9a^2-10ab+5b^2)\gamma^3\\
&-a^4b^4(a+b)(29a-15b)\gamma^2-6a^5b^5(a+b)\gamma+a^6b^6=0;
\end{align*}

\item the caustic is an ellipse, with $\gamma \in (-b,0)$ or a hyperbola and $D_2D_4-D^{2}_3=0$, i.e.
\begin{align*}
&(a^2-10ab+5b^2)(a+b)^4\gamma^6+2ab(3a-5b)(a+b)^4\gamma^5\\
&+a^2b^2(a+b)(15a^3+5a^2b+45ab^2-9b^3)\gamma^4+4a^3b^3(a+b)(5a^2-10ab+9b^2)\gamma^3\\
&+a^4b^4(a+b)(15a-29b)\gamma^2+6a^5b^5(a+b)\gamma+a^6b^6=0.
\end{align*}
\end{itemize}
Each real solution $\gamma$ for the above equations for some fixed values of $a$ and $b$ will produce a $5$-elliptic periodic trajectory which is not $5$-periodic.
Some examples are shown in Figure \ref{fig:5-elliptic-periodic}.
	\begin{figure}[h]
\begin{minipage}{0.5\textwidth}
		\centering
		\begin{tikzpicture}[scale=1]
\clip (-4,-4) rectangle (4,4);


\draw[thick](0,0) circle [x radius={sqrt(7)}, y radius={sqrt(4)}];

\draw[thick,dashed,gray](0,0) circle [x radius={sqrt(7+3.3848)}, y radius={sqrt(4-3.3848)}];

\draw[gray] ({sqrt(11)+2},-2) -- (-2,{sqrt(11)+2});
\draw[gray] ({-sqrt(11)-2},2) -- (2,{-sqrt(11)-2});
\draw[gray] (2,{sqrt(11)+2}) -- ({-sqrt(11)-2},-2);
\draw[gray] (-2,{-sqrt(11)-2}) -- ({sqrt(11)+2},2);

\draw[very thick, blue](1, -1.852) -- (2.479, -.699) -- (-2.373, -.8846) -- (-1.613, -1.585) -- (2.569, -.4789) -- (-1, -1.852);
\end{tikzpicture}
\end{minipage}
\begin{minipage}{0.5\textwidth}
	\centering
	\begin{tikzpicture}[scale=1]
\clip (-4,-4) rectangle (4,4);

\draw[thick](0,0) circle [x radius={sqrt(3)}, y radius={sqrt(7)}];

hyperbola lambda=3.4462
\draw[domain=-1:1,smooth,thick,dashed,variable=\t] plot ({\t},{sqrt(7+3.4462)*sqrt(1-\t*\t/(3-3.4462)});
\draw[domain=-1:1,smooth,thick,dashed,variable=\t] plot ({\t},{-sqrt(7+3.4462)*sqrt(1-\t*\t/(3-3.4462)});

\draw[gray] ({sqrt(10)+2},-2) -- (-2,{sqrt(10)+2});
\draw[gray] ({-sqrt(10)-2},2) -- (2,{-sqrt(10)-2});
\draw[gray] (2,{sqrt(10)+2}) -- ({-sqrt(10)-2},-2);
\draw[gray] (-2,{-sqrt(10)-2}) -- ({sqrt(10)+2},2);

\draw[very thick, blue] (.5, -2.53311) -- (-.543390, 2.51218) -- (-1.43296, 1.48620) --(0.0130805, -2.64568) -- (1.49069, 1.34723) -- (.5, 2.53311);
\end{tikzpicture}
\end{minipage}
		\caption{$5$-elliptic periodic trajectories. On the left, the caustic is an ellipse ($a=7$, $b=4$, $\gamma\approx-3.3848$) and a hyperbola on the right ($a=3$, $b=7$, $\gamma\approx3.4462$).}
	\label{fig:5-elliptic-periodic}
\end{figure}
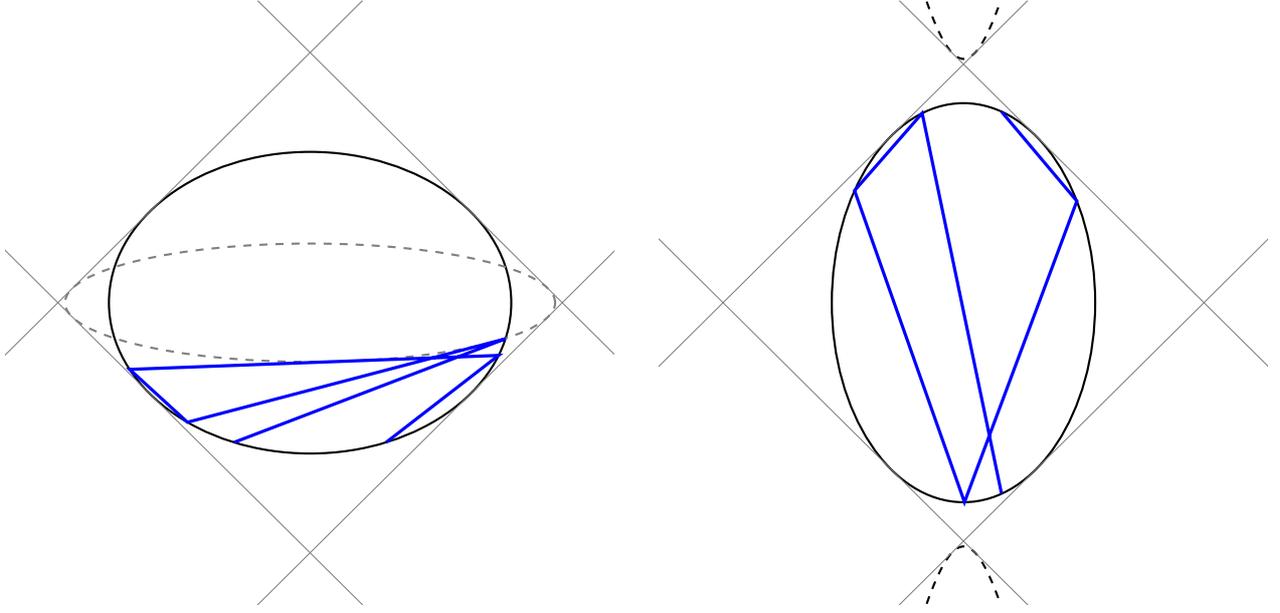

\subsubsection*{Discriminantly separable polynomials and elliptic periodicity}

Since the case $n=2$ is trivial, we start with the case $n=3$.

From  \refeq{eq:3-elliptic-periodic1} and \refeq{eq:3-elliptic-periodic2}, we have:
\begin{gather*}
\mathsf{G}_{1}(a,b,\gamma)=-(a+b)(3a-b)\gamma^2-2ab(a+b)\gamma+a^2b^2,
\\
\mathsf{G}_{2}(a,b,\gamma)=	(a+b)(a-3b)\gamma^2+2ab(a+b)\gamma+a^2b^2,
\end{gather*}
and we calculate the discriminants, which factorize as follows:
$$	
\mathcal{D_{\gamma}}\mathsf{G}_{1}=16a^3b^2(a+b),
\quad
	\mathcal{D_{\gamma}}\mathsf{G}_{2}=16b^3a^2(a+b).
$$

Similarly, for $n=4$, we have:
\begin{align*}
\mathsf{G}_{3}(a,b,\gamma)=&(a+b)^4\gamma^4-4ab(a+b)(a-b)^2\gamma^3-2a^2b^2(a+b)(5a-3b)\gamma^2-4a^3b^3(a+b)\gamma+a^4b^4,
\\
\mathsf{G}_{4}(a,b,\gamma)=&(a+b)^4\gamma^4+4ab(a+b)(a-b)^2\gamma^3+2a^2b^2(a+b)(3a-5b)\gamma^2+4a^3b^3(a+b)\gamma+a^4b^4,
\\
\mathsf{G}_{5}(a,b,\gamma)=&(a^2-6ab+b^2)(a+b)^2\gamma^4+4ab(a-b)(a+b)^2\gamma^3
+2a^2b^2(3a^2+2ab+3b^2)\gamma^2
\\&
+4a^3b^3(a-b)\gamma+a^4b^4.
\end{align*}
The discriminants of these polynomials factorize as follows:
\begin{gather*}
	\mathcal{D_{\gamma}}\mathsf{G}_{3}=-2^{16}a^{16}b^{14}(8a^2+8ab+27b^2)(a+b)^4,
\\	
\mathcal{D_{\gamma}}\mathsf{G}_{4}=-2^{16}a^{14}b^{16}(27a^2+8ab+8b^2)(a+b)^4,
\\
\mathcal{D_{\gamma}}\mathsf{G}_{5}=2^{12}(32a^6-491a^5b-439a^4b^2+194a^3b^3-62a^2b^4-39ab^5+5b^6)(a+b)^3b^{15}a^{12}.
\end{gather*}

Using the transformation $(a,b)\mapsto (a,\hat{b})$, where $\hat{b}=\dfrac{b}{a}$, we get:
\begin{align*}
&\mathcal{D_{\gamma}}\mathsf{G}_{1}=16a^6\hat{b}^2(1+\hat{b}),
\\	
&\mathcal{D_{\gamma}}\mathsf{G}_{2}=16a^5\hat{b}^3(1+\hat{b}),
\\	&\mathcal{D_{\gamma}}\mathsf{G}_{3}=-2^{16}a^{36}\hat{b}^{14}(8+8\hat{b}+27\hat{b}^2)(1+\hat{b})^4,
\\
&\mathcal{D_{\gamma}}\mathsf{G}_{4}=-2^{16}a^{36}\hat{b}^{16}(27+8\hat{b}+8\hat{b}^2)(1+\hat{b})^4,
\\
&\mathcal{D_{\gamma}}\mathsf{G}_{5}=2^{12}a^{36}\hat{b}^{15}(32-491\hat{b}-439\hat{b}^2+194\hat{b}^3-62\hat{b}^4-39\hat{b}^5+5\hat{b}^6)(1+\hat{b})^3.
\end{align*}

\section{Polynomial equations}
\label{sec:polynomial}

Now we want to express the periodicity conditions for billiard trajectories in the Minkowski plane in terms of polynomial functions equations.
\begin{theorem}\label{th:polynomial}
The billiard trajectories within $\E$ with caustic $\C_{\gamma}$ are $n$-periodic if and only if there exists a pair of real polynomials $p_{d_1}$, $q_{d_2}$ of degrees $d_1$, $d_2$ respectively, and satisfying the following:
	\begin{itemize}
		\item[(a)] if $n=2m$ is even,  then $d_1=m$, $d_2=m-2$, and
		$$
		p_{m}^2(s)
		-
		s\left(s-\frac1a\right)\left(s+\frac1b\right)\left( s-\frac1{\gamma}\right)
		{q}_{m-2}^2(s)=1;
		$$
		\item[(b)] if $n=2m+1$ is odd, then $d_1=m$, $d_2=m-1$, and
		$$
		\left( s-\frac1{\gamma}\right)p_m^2(s)
		-
		s\left(s-\frac1a\right)\left(s+\frac1b\right)q_{m-1}^2(s)=-\sign\gamma.
		$$	
	\end{itemize}	
\end{theorem}
\begin{proof}
We note first that the  proof of Theorem \ref{th:cayley-billiard} implies that there is a non-trivial linear combination of the bases \eqref{eq:basis-even} for $n$ even, or \eqref{eq:basis-odd} for $n$ odd,  with the zero of order $n$ at $x=0$.
	
	(a) For $n=2m$, from there we get that there are real polynomials $p_m^*(x)$ and $q_{m-2}^*(x)$ of degrees $m$ and $m-2$ respectively, such that the expression
	$$
	p_{m}^*(x)-q_{m-2}^*(x)\sqrt{\varepsilon(a-x)(b+x)(\gamma-x)}
	$$
	has a zero of order $2m$ at $x=0$.
	Multiplying that expression by
	$$
	p_{m}^*(x)+q_{m-2}^*(x)\sqrt{\varepsilon(a-x)(b+x)(\gamma-x)},
	$$
	we get that the polynomial $(p_{m}^*(x))^2-\varepsilon(a-x)(b+x)(\gamma-x)(q_{m-2}^*(x))^{2}$ has a zero of order $2m$ at $x=0$.
	Since the degree of that polynomial is $2m$, is follows that:
	$$
	(p_{m}^*(x))^2-\varepsilon(a-x)(b+x)(\gamma-x)(q_{m-2}^*(x))^{2}=cx^{2m},
	$$
	for some constant $c$.
	Notice that $c$ is positive, since it equals the square of the leading coefficient of $p_m^*$.
	Dividing the last relation by $cx^{2m}$ and introducing $s=1/x$, we get the requested relation.
	
	(b) On the other hand, for $n=2m+1$, we get
	that there are real polynomials $p_m^*(x)$ and $q_{m-1}^*(x)$ of degrees $m$ and $m-1$ respectively, such that the expression
	$$
	p_{m}^*(x)-q_{m-1}^*(x)\frac{\sqrt{\varepsilon(a-x)(b+x)(\gamma-x)}}{\gamma-x}
	$$
	has a zero of order $2m+1$ at $x=0$.
	Multiplying that expression by
	$$
	(\gamma-x)
	\left(
	p_{m}^*(x)+q_{m-1}^*(x)\frac{\sqrt{\varepsilon(a-x)(b+x)(\gamma-x)}}{\gamma-x}
	\right)
	,
	$$
	we get that the polynomial $(\gamma-x)(p_{m}^*(x))^2-\varepsilon(a-x)(b+x)(q_{m-1}^*(x))^{2}$ has a zero of order $2m+1$ at $x=0$.
	Since the degree of that polynomial is $2m+1$, is follows that:
	$$
	(\gamma-x)(p_{m}^*(x))^2-\varepsilon(a-x)(b+x)(q_{m-1}^*(x))^{2}=cx^{2m+1},
	$$
	for some constant $c$.
	Notice that $c$ is negative, since it equals the opposite of the square of the leading coefficient of $p_m^*$.
	Dividing the last relation by $-\varepsilon cx^{2m+1}$ and introducing $s=1/x$, we get the requested relation.
\end{proof}	

\begin{corollary}\label{cor:pell-periodic}
	If the billiard trajectories within $\E$ with caustic $\C_{\gamma}$ are $n$-periodic, then there exist real polynomials $\hat{p}_n$ and $\hat{q}_{n-2}$ of degrees $n$ and $n-2$ respectively, which satisfy the Pell  equation:
	\begin{equation}\label{eq:pell}
	\hat{p}_{n}^2(s)-s\left(s-\frac1a\right)\left(s+\frac1b\right)\left( s-\frac1{\gamma}\right)\hat{q}_{n-2}^2(s)=1.
	\end{equation}
\end{corollary}
\begin{proof}
	For $n=2m$, take $\hat{p}_n=2p_{m}^2-1$ and $\hat{q}_{n-2}=2p_mq_{m-2}$.
	For $n=2m+1$, we set $\hat{p}_n=2\left(\gamma s-1\right)p_{m}^2+\sign\gamma$ and $\hat{q}_{n-2}=2p_mq_{m-1}$.
\end{proof}

\begin{theorem}\label{th:polynomial-elliptic}
	The billiard trajectories within $\E$ with caustic $\C_{\gamma}$ are elliptic $n$-periodic without being $n$-periodic if and only if there exists a pair of real polynomials $p_{d_1}$, $q_{d_2}$ of degrees $d_1$, $d_2$ respectively, and satisfying the following:
	\begin{itemize}
		\item[(a)] $\C_{\gamma}$ is an ellipse, $0<\gamma<a$, and
		\begin{itemize}
			\item $n=2m$ is even, $d_1=d_2=m-1$,
			$$
			s\left(s-\frac1a\right)p_{m-1}^2(s)
			-\left(s+\frac1b\right)\left( s-\frac1{\gamma}\right)q_{m-1}^2(s)=1;
			$$
			\item $n=2m+1$ is odd, $ d_1=m$, $d_2=m-1$,
			$$
			\left(s+\frac1b\right) p_{m}^2(s)
			-s\left(s-\frac1a\right)\left( s-\frac1{\gamma}\right)q_{m-1}^2(s)=1;
			$$
		\end{itemize}

\item[(b)] $\C_{\gamma}$ is an ellipse, $-b<\gamma<0$, and
\begin{itemize}
	\item $n=2m$ is even, $d_1=d_2=m-1$,
	$$
	s\left(s+\frac1b\right)p_{m-1}^2(s)
	-\left(s-\frac1a\right)\left( s-\frac1{\gamma}\right)q_{m-1}^2(s)=1;
	$$
	\item $n=2m+1$ is odd, $ d_1=m$, $d_2=m-1$,
	$$
	\left(s-\frac1a\right) p_{m}^2(s)
	-s\left(s+\frac1b\right)\left(s-\frac{1}{\gamma}\right)q_{m-1}^2(s)=-1;
	$$
\end{itemize}
		
		\item[(c)] $\C_{\gamma}$ is a hyperbola and $n=2m$ is even, $d_1=d_2=m-1$,
			$$
			\left( s-\frac{1}{\gamma}\right)p_{m-1}^2(s)
			-s\left(s-\frac1a\right)\left(s+\frac1b\right)q_{m-1}^2(s)=-\sign\gamma;
			$$

		\item[(d)] $\C_{\gamma}$ is a hyperbola, $n=2m+1$ is odd, $d_1=m$, $d_2=m-1$,
		$$
		\left(s-\frac1a\right) p_{m}^2(s)
		-s\left(s+\frac1b\right)\left(s-\frac1{\gamma}\right)q_{m-1}^2(s)=-1;
		$$
	\item[(e)] $\C_{\gamma}$ is a hyperbola, $n=2m+1$ is odd, $d_1=m$, $d_2=m-1$,
$$
\left(s+\frac1b\right) p_{m}^2(s)
-s\left(s-\frac1a\right)\left(s-\frac1{\gamma}\right)q_{m-1}^2(s)=1.
$$
	\end{itemize}	
\end{theorem}
\begin{proof}
	(a) For $n=2m$, the proof of Theorem \ref{th:elliptic-cayley} implies that there are polynomials $p_{m-1}^*(x)$ and $q_{m-1}^*(x)$ of degrees $m-1$, such that the expression
	$$
	p_{m-1}^*(x)-q_{m-1}^*(x)\frac{\sqrt{(a-x)(b+x)(\gamma-x)}}{a-x}
	$$
	has a zero of order $2m$ at $x=0$.
	Multiplying that expression by
	$$
	(a-x)\left(p_{m-1}^*(x)+q_{m-1}^*(x)\frac{\sqrt{(a-x)(b+x)(\gamma-x)}}{a-x}\right),
	$$
	we get that the polynomial
	$(a-x)(p_{m-1}^*(x))^2-(a-x)(b+x)(q_{m-1}^*(x))^2$ has a zero of order $2m$ at $x=0$.
	Since the degree of that polynomial is $2m$, is follows that:
	$$
	(a-x)(p_{m-1}^*(x))^2-(b+x)(\gamma-x)(q_{m-1}^*(x))^2=cx^{2m},
	$$
	for some constant $c$.
	Notice that $c$ is positive, since it equals the square of the leading coefficient of $q_{m-1}^*$.
	Dividing the last relation by $cx^{2m}$ and introducing $s=1/x$, we get the requested relation.
	
	For $n=2m+1$, the proof of Theorem \ref{th:elliptic-cayley} implies that there are polynomials $p_{m}^*(x)$ and $q_{m-1}^*(x)$ of degrees $m$ and $m-1$, such that the expression
	$$
	p_{m}^*(x)-q_{m-1}^*(x)\frac{\sqrt{(a-x)(b+x)(\gamma-x)}}{b+x}
	$$
	has a zero of order $2m+1$ at $x=0$.
	Multiplying that expression by
	$$
	(b+x)\left(p_{m}^*(x)+q_{m-1}^*(x)\frac{\sqrt{(a-x)(b+x)(\gamma-x)}}{b+x}\right)
	,
	$$
	we get that the polynomial
	$(b+x)(p_{m}^*(x))^2-(a-x)(\gamma-x)(q_{m-1}^*(x))^2$ has a zero of order $2m+1$ at $x=0$.
	Since the degree of that polynomial is $2m+1$, is follows that:
	$$
	(b+x)(p_{m}^*(x))^2-(a-x)(\gamma-x)(q_{m-1}^*(x))^2=cx^{2m+1}
	$$
	for some constant $c$.
	Notice that $c$ is positive, since it equals the square of the leading coefficient of $p_{m}^*$.
	Dividing the last relation by $cx^{2m+1}$ and introducing $s=1/x$, we get the requested relation.
	
(b) For $n=2m$, the proof of Theorem \ref{th:elliptic-cayley} implies that there are real polynomials $p_{m-1}^*(x)$ and $q_{m-1}^*(x)$ of degrees $m-1$, such that the expression
	$$
	p_{m-1}^*(x)-q_{m-1}^*(x)\frac{\sqrt{-(a-x)(b+x)(\gamma-x)}}{b+x}
	$$
	has a zero of order $2m$ at $x=0$.
	Multiplying that expression by
	$$
	(b+x)\left(p_{m-1}^*(x)+q_{m-1}^*(x)\frac{\sqrt{-(a-x)(b+x)(\gamma-x)}}{b+x}\right),
	$$
	we get that the polynomial
	$(b+x)(p_{m-1}^*(x))^2+(a-x)(\gamma-x)(q_{m-1}^*(x))^2$ has a zero of order $2m$ at $x=0$.
	Since the degree of that polynomial is $2m$, is follows that:
	$$
	(b+x)(p_{m-1}^*(x))^2+(a-x)(\gamma-x)(q_{m-1}^*(x))^2=cx^{2m},
	$$
	for some constant $c$.
	Notice that $c$ is positive, since it equals to the square of the leading coefficient of $q_{m-1}^*$.
	Dividing the last relation by $cx^{2m}$ and introducing $s=1/x$, we get the requested relation.
	
For $n=2m+1$, the proof of Theorem \ref{th:elliptic-cayley} implies that there are polynomials $p_{m}^*(x)$ and $q_{m-1}^*(x)$ of degrees $m$ and $m-1$, such that the expression
	$$
	p_{m}^*(x)-q_{m-1}^*(x)\frac{\sqrt{-(a-x)(b+x)(\gamma-x)}}{a-x}
	$$
	has a zero of order $2m+1$ at $x=0$.
	Multiplying that expression by
	$$
	(a-x)\left(p_{m}^*(x)+q_{m-1}^*(x)\frac{\sqrt{-(a-x)(b+x)(\gamma-x)}}{a-x}\right)
	,
	$$
	we get that the polynomial
	$(a-x)(p_{m}^*(x))^2+(b+x)(\gamma-x)(q_{m-1}^*(x))^2$ has a zero of order $2m+1$ at $x=0$.
	Since the degree of that polynomial is $2m+1$, is follows that:
	$$
	(a-x)(p_{m}^*(x))^2+(b+x)(\gamma-x)(q_{m-1}^*(x))^2=cx^{2m+1}
	$$
	for some constant $c$.
	Notice that $c$ is negative, since it is opposite to the square of the leading coefficient of $p_{m}^*$.
	Dividing the last relation by $-cx^{2m+1}$ and introducing $s=1/x$, we get the requested relation.
	
For (c), the proof of Theorem \ref{th:elliptic-cayley} implies that there are polynomials real $p_{m}^*(x)$ and $q_{m-1}^*(x)$ of degrees $m$ and $m-1$, such that the expression
	$$
	p_{m}^*(x)-q_{m-1}^*(x)\frac{\sqrt{\varepsilon(a-x)(b+x)(\gamma-x)}}{\gamma-x}
	$$
	has a zero of order $2m+1$ at $x=0$.
	Multiplying that expression by
	$$
	(\gamma-x)\left(p_{m}^*(x)+q_{m-1}^*(x)\frac{\sqrt{\varepsilon(a-x)(b+x)(\gamma-x)}}{\gamma-x}\right)
	,
	$$
	we get that the polynomial
	$(\gamma-x)(p_{m}^*(x))^2-\varepsilon(a-x)(b+x)(q_{m-1}^*(x))^2$ has a zero of order $2m+1$ at $x=0$.
	Since the degree of that polynomial is $2m+1$, is follows that:
	$$
	(\gamma-x)(p_{m}^*(x))^2-\varepsilon(a-x)(b+x)(q_{m-1}^*(x))^2=cx^{2m+1}
	$$
	for some constant $c$.
	Notice that $c$ is negative, since it is opposite to the square of the leading coefficient of $p_{m}^*$.
	Dividing the last relation by $-\varepsilon cx^{2m+1}$ and introducing $s=1/x$, we get the requested relation.

(d)
The proof of Theorem \ref{th:elliptic-cayley} implies that there are real polynomials $p_{m}^*(x)$ and $q_{m-1}^*(x)$ of degrees $m$ and $m-1$, such that the expression
$$
p_{m}^*(x)-q_{m-1}^*(x)\frac{\sqrt{\varepsilon(a-x)(b+x)(\gamma-x)}}{a-x}
$$
has a zero of order $2m+1$ at $x=0$.
Multiplying that expression by
$$
(a-x)\left(p_{m}^*(x)+q_{m-1}^*(x)\frac{\sqrt{\varepsilon(a-x)(b+x)(\gamma-x)}}{a-x}\right)
,
$$
we get that the polynomial
$(a-x)(p_{m}^*(x))^2-\varepsilon(b+x)(\gamma-x)(q_{m-1}^*(x))^2$ has a zero of order $2m+1$ at $x=0$.
Since the degree of that polynomial is $2m+1$, is follows that:
$$
(a-x)(p_{m}^*(x))^2-\varepsilon(b+x)(\gamma-x)(q_{m-1}^*(x))^2=cx^{2m+1}
$$
for some constant $c$.
Notice that $c$ is negative, since it is opposite to the square of the leading coefficient of $p_{m}^*$.
Dividing the last relation by $-cx^{2m+1}$ and introducing $s=1/x$, we get the requested relation.

(e) For $n=2m+1$, the proof of Theorem \ref{th:elliptic-cayley} implies that there are real polynomials $p_{m}^*(x)$ and $q_{m-1}^*(x)$ of degrees $m$ and $m-1$, such that the expression
$$
p_{m}^*(x)-q_{m-1}^*(x)\frac{\sqrt{\varepsilon(a-x)(b+x)(\gamma-x)}}{b+x}
$$
has a zero of order $2m+1$ at $x=0$.
Multiplying that expression by
$$
(b+x)\left(p_{m}^*(x)+q_{m-1}^*(x)\frac{\sqrt{\varepsilon(a-x)(b+x)(\gamma-x)}}{b+x}\right)
,
$$
we get that the polynomial
$(b+x)(p_{m}^*(x))^2-\varepsilon(a-x)(\gamma-x)(q_{m-1}^*(x))^2$ has a zero of order $2m+1$ at $x=0$.
Since the degree of that polynomial is $2m+1$, is follows that:
$$
(b+x)(p_{m}^*(x))^2-\varepsilon(a-x)(\gamma-x)(q_{m-1}^*(x))^2=cx^{2m+1}
$$
for some constant $c$.
Notice that $c$ is positive, since it equals the square of the leading coefficient of $p_{m}^*$.
Dividing the last relation by $cx^{2m+1}$ and introducing $s=1/x$, we get the requested relation.
\end{proof}	
After corollary \ref{cor:pell-periodic} and relation \refeq{eq:pell}, we see that the Pell equations arise as the functional polynomial conditions for periodicity. Let us recall some important properties of the solutions of the Pell equations.

\section{Classical Extremal Polynomials and Caustics}
\label{sec:extremal}

\subsection{Fundamental Properties of Extremal Polynomials}
From the previous section we know that the Pell equation plays a key role in functional-polynomial formulation of periodicity conditions in the Minkowski plane. The solutions of the Pell equation are so-called extremal polynomials.
Denote $\{c_1,c_2,c_3,c_4\}=\{0, 1/a, -1/b, 1/\gamma\}$ with the ordering $c_{1}<c_{2}<c_{3}<c_{4}$.
The polynomials $\hat {p}_n$ are so called \emph{generalized Chebyshev polynomials} on two intervals $[c_1, c_2]\cup [c_3, c_4]$, with an appropriate normalization. Namely, one can consider the question of finding the monic polynomial of certain degree  $n$ which minimizes the maximum norm on the union of two intervals. Denote such a polynomial as $\hat P_n$ and its norm $L_n$. The fact that polynomial $\hat {p}_n$ is a solution of the Pell equation on the union of intervals $[c_1, c_2]\cup [c_3, c_4]$ is equivalent to the
following conditions:
\begin{itemize}
	\item[(i)] $\hat {p}_n=\hat {P}_n/\pm L_n$
	\item[(ii)] the set $[c_1, c_2]\cup [c_3, c_4]$ is the maximal subset of $\mathbf R$ for which $\hat {P}_n$
	is the minimal polynomial in the sense above.
\end{itemize}

Chebyshev was the first who considered a similar problem on one interval, and this was how celebrated Chebyshev polynomials emerged in XIXth century.
Let us recall a fundamental result about generalized Chebyshev polynomials \cite{AhiezerAPPROX, Akh4}.

\begin{theorem}[A corollary of the Krein-Levin-Nudelman Theorem \cite{KLN1990}]
	A polynomial  $P_{n}$ of degree $n$ satisfies a Pell equation on the union of intervals $[c_{1}, c_{2}]\cup[c_{3}, c_{4}]$ if and only if there exists an integer $n_{1}$ such that the equation holds:
	\begin{equation}\label{eq:KLN}
	n_{1}\int_{c_{2}}^{c_{3}}\hat f(s)ds=n\int_{c_{4}}^{\infty}\hat f(s)ds, \quad \hat f(s)=\frac{C}{\sqrt{\prod_{i=1}^4(s-c_i)}}.
	\end{equation}
(Here $C$ is a nonessential constant.)
	The modulus of the polynomial reaches its maximal values $L_{n}$ at the points $c_{i}:$ $|P_{n}(c_{i})|=L_{n}$.\\
	In addition, there are exactly $\tau_{1}=n-n_{1}-1$ internal extremal points of the interval $[c_{3}, c_{4}]$ where $|P_{n}|$ reaches the value $L_{n}$, and there are $\tau_{2}=n_{1}-1$ internal extremal points of $[c_{1}, c_{2}]$ with the same property.
\end{theorem}
\begin{definition}[\cite{RR2014,DragRadn2018,DragRadn2019rcd}]
	We call the pair $(n, n_1)$ the partition and $(\tau_1, \tau_2)$ the signature of the generalized Chebyshev polynomial $P_n$.
\end{definition}
Now we are going to formulate and prove the main result of this Section, which relates $n_1, n_2$ the numbers of reflections off relativistic ellipses and off relativistic hyperbolas respectively with the partition and the signature of the related solution of a Pell equation.

\begin{theorem}\label{th:impactwinding}
	Given a periodic billiard trajectory with period $n=n_{1}+n_{2}$, where $n_{1}$ is the number of reflections off relativistic ellipses, $n_{2}$ the number of reflections off the relativistic hyperbolas, then the partition corresponding to this trajectory is $(n, n_{1})$. The corresponding extremal polynomial $\hat{p}_{n}$ of degree $n$ has $n_{1}-1$ internal extremal points in the first interval and $n-n_{1}-1=n_{2}-1$ internal extremal points in the second interval.
\end{theorem}

\begin{proof}
	Recall that $c_{1}<c_{2}<c_{3}<c_{4}$.
	From the equation \refeq{InteEqn}, one has:
	\begin{equation}\label{IntEq2}
	n_{1}\int_{\alpha_{1}}^{0}f(x)dx+n_{2}\int_{\beta_{1}}^{0}f(x)dx=0
	\end{equation}
	where
	$\alpha_{1}$ is the largest negative value in $\{a, -b, \gamma\}$ and $\beta_{1}$ the smallest positive value in $\{a, -b, \gamma\}$.

The proof decomposes to the following cases:
\begin{itemize}
	\item Case 1: $\C_{\gamma}$ is an ellipse and $\gamma <0$, shown in Figure \ref{fig:case1};
	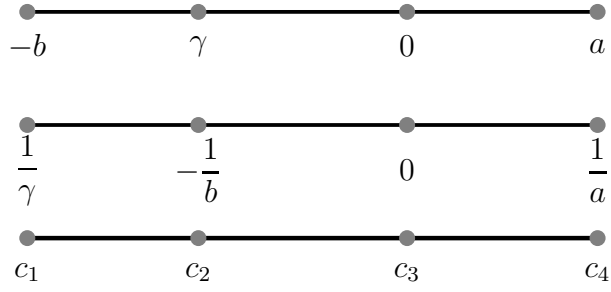
\begin{figure}[h]
	\centering
	\begin{tikzpicture}[scale=1.5]
	\draw[line width=.5mm, black] (0,1)--(5,1) ;
	\draw[line width=.5mm, black] (0,0)--(5,0);
	\draw[line width=.6mm, black] (0,-1)--(5,-1);
	
	\fill[gray] (0, 1) circle (2pt);
	\fill[gray] (5, 1) circle (2pt);
	\fill[gray] (1.5, 1) circle (2pt);
	\fill[gray] (3.33, 1) circle (2pt);
	
	\fill[gray] (0, 0) circle (2pt);
	\fill[gray] (5, 0) circle (2pt);
	\fill[gray] (1.5, 0) circle (2pt);
	\fill[gray] (3.33, 0) circle (2pt);
	
	\fill[gray] (0, -1) circle (2pt);
	\fill[gray] (5, -1) circle (2pt);
	\fill[gray] (1.5, -1) circle (2pt);
	\fill[gray] (3.33, -1) circle (2pt);
	
	\draw (0,.7)  node {$-b$};
	\draw (5,.7)  node {${a}$};
	\draw (1.5,.7)  node {${\gamma}$};
	\draw (3.33,.7)  node {${0}$};
	
	\draw (0,-.4)  node {$\dfrac{1}{{\gamma}}$};
	\draw (5,-.4)  node {$\dfrac{1}{{a}}$};
	\draw (1.5,-.4)  node {$-\dfrac{1}{{b}}$};
	\draw (3.33,-.4)  node {$0$};
	
	\draw (0,-1.3)  node {${c}_{1}$};
	\draw (5,-1.3)  node {${c}_{4}$};
	\draw (1.5,-1.3)  node {${c}_{2}$};
	\draw (3.33,-1.3)  node {${c}_{3}$};

	\end{tikzpicture}
	\caption{Case 1: $\alpha_{1}=\gamma$, $\beta_{1}=a$}
	\label{fig:case1}
\end{figure}

	\item Case 2: $\C_{\gamma}$ is an ellipse and $\gamma >0$, shown in Figure \ref{fig:case2};
	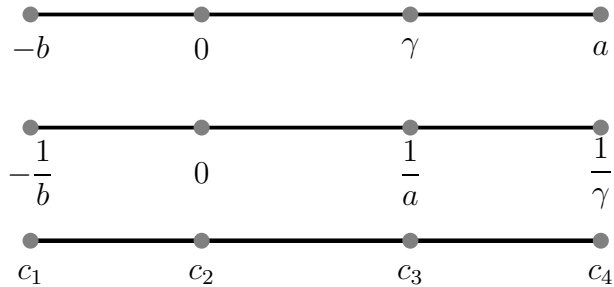
\begin{figure}[H]
	\centering
		\begin{tikzpicture}[scale=1.5]
\draw[line width=.5mm, black] (0,1)--(5,1) ;
\draw[line width=.5mm, black] (0,0)--(5,0);
\draw[line width=.6mm, black] (0,-1)--(5,-1);

\fill[gray] (0, 1) circle (2pt);
\fill[gray] (5, 1) circle (2pt);
\fill[gray] (1.5, 1) circle (2pt);
\fill[gray] (3.33, 1) circle (2pt);

\fill[gray] (0, 0) circle (2pt);
\fill[gray] (5, 0) circle (2pt);
\fill[gray] (1.5, 0) circle (2pt);
\fill[gray] (3.33, 0) circle (2pt);

\fill[gray] (0, -1) circle (2pt);
\fill[gray] (5, -1) circle (2pt);
\fill[gray] (1.5, -1) circle (2pt);
\fill[gray] (3.33, -1) circle (2pt);

\draw (0,.7)  node {$-{b}$};
\draw (5,.7)  node {${a}$};
\draw (1.5,.7)  node {$0$};
\draw (3.33,.7)  node {${\gamma}$};

\draw (0,-.4)  node {$-\dfrac{1}{{b}}$};
\draw (5,-.4)  node {$\dfrac{1}{{\gamma}}$};
\draw (1.5,-.4)  node {$0$};
\draw (3.33,-.4)  node {$\dfrac{1}{{a}}$};

\draw (0,-1.3)  node {${c}_{1}$};
\draw (5,-1.3)  node {${c}_{4}$};
\draw (1.5,-1.3)  node {${c}_{2}$};
\draw (3.33,-1.3)  node {${c}_{3}$};

\end{tikzpicture}
	\caption{Case 2: $\beta_{1}=\gamma$, $\alpha_{1}=-b$}
	\label{fig:case2}
\end{figure}

\item Case 3(i): $\C_{\gamma}$ is a hyperbola and $\gamma <-b$, shown in Figure \ref{fig:case3i};
	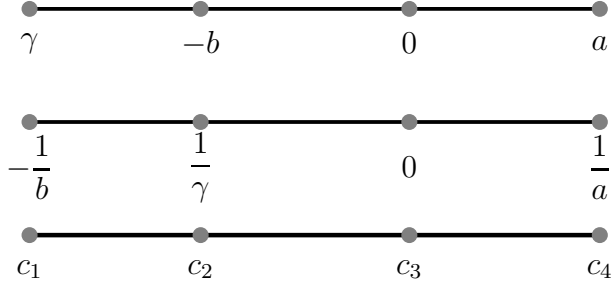
\begin{figure}[h]
	\centering
		\begin{tikzpicture}[scale=1.5]

\draw[line width=.5mm, black] (0,1)--(5,1) ;
\draw[line width=.5mm, black] (0,0)--(5,0);
\draw[line width=.6mm, black] (0,-1)--(5,-1);

\fill[gray] (0, 1) circle (2pt);
\fill[gray] (5, 1) circle (2pt);
\fill[gray] (1.5, 1) circle (2pt);
\fill[gray] (3.33, 1) circle (2pt);

\fill[gray] (0, 0) circle (2pt);
\fill[gray] (5, 0) circle (2pt);
\fill[gray] (1.5, 0) circle (2pt);
\fill[gray] (3.33, 0) circle (2pt);

\fill[gray] (0, -1) circle (2pt);
\fill[gray] (5, -1) circle (2pt);
\fill[gray] (1.5, -1) circle (2pt);
\fill[gray] (3.33, -1) circle (2pt);

\draw (0,.7)  node {${\gamma}$};
\draw (5,.7)  node {${a}$};
\draw (1.5,.7)  node {$-{b}$};
\draw (3.33,.7)  node {$0$};

\draw (0,-.4)  node {$-\dfrac{1}{{b}}$};
\draw (5,-.4)  node {$\dfrac{1}{{a}}$};
\draw (1.5,-.4)  node {$\dfrac{1}{{\gamma}}$};
\draw (3.33,-.4)  node {$0$};

\draw (0,-1.3)  node {${c}_{1}$};
\draw (5,-1.3)  node {${c}_{4}$};
\draw (1.5,-1.3)  node {${c}_{2}$};
\draw (3.33,-1.3)  node {${c}_{3}$};

\end{tikzpicture}
	\caption{Case 3(i): $\beta_{1}=a$, $\alpha_{1}=-b$}
	\label{fig:case3i}
\end{figure}

\item Case 3(ii): $\C_{\gamma}$ is a hyperbola and $\gamma >a$, shown in Figure \ref{fig:case3ii}.
	\begin{figure}[h]
	\centering
		\begin{tikzpicture}[scale=1.5]

\draw[line width=.5mm, black] (0,1)--(5,1) ;
\draw[line width=.5mm, black] (0,0)--(5,0);
\draw[line width=.6mm, black] (0,-1)--(5,-1);

\fill[gray] (0, 1) circle (2pt);
\fill[gray] (5, 1) circle (2pt);
\fill[gray] (1.5, 1) circle (2pt);
\fill[gray] (3.33, 1) circle (2pt);

\fill[gray] (0, 0) circle (2pt);
\fill[gray] (5, 0) circle (2pt);
\fill[gray] (1.5, 0) circle (2pt);
\fill[gray] (3.33, 0) circle (2pt);

\fill[gray] (0, -1) circle (2pt);
\fill[gray] (5, -1) circle (2pt);
\fill[gray] (1.5, -1) circle (2pt);
\fill[gray] (3.33, -1) circle (2pt);

\draw (0,.7)  node {$-{b}$};
\draw (5,.7)  node {${\gamma}$};
\draw (1.5,.7)  node {$0$};
\draw (3.33,.7)  node {${a}$};

\draw (0,-.4)  node {$-\dfrac{1}{{b}}$};
\draw (5,-.4)  node {$\dfrac{1}{{a}}$};
\draw (1.5,-.4)  node {$0$};
\draw (3.33,-.4)  node {$\dfrac{1}{{\gamma}}$};

\draw (0,-1.3)  node {${c}_{1}$};
\draw (5,-1.3)  node {${c}_{4}$};
\draw (1.5,-1.3)  node {${c}_{2}$};
\draw (3.33,-1.3)  node {${c}_{3}$};
\end{tikzpicture}
	\caption{Case 3(ii): $\beta_{1}=a$, $\alpha_{1}=-b$}
	\label{fig:case3ii}
\end{figure}
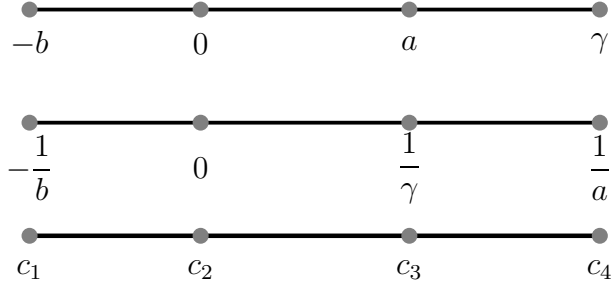

\end{itemize}	
We provide proof in the Case 1.
The proofs for other cases are analogous.	

Equation \refeq{IntEq2} is equivalent to
\begin{align*}
0
&=
n_{1}\int_{\gamma}^{0}f(x)dx + n_{1}\int_{0}^{a}f(x)dx+n_{2}\int_{a}^{0}f(x)dx-n_{1}\int_{0}^{a}f(x)dx
\\&
=
n_{1}\int_{\gamma}^{a}f(x)dx+(n_{1}+n_{2})\int_{a}^{0}f(x)dx,
\end{align*}
thus
$$
n_{1}\int_{\gamma}^{a}f(x)dx=(n_{1}+n_{2})\int_{0}^{a}f(x)dx.
$$	
Since the cycles around the cuts on the elliptic curve are homologous:
	$$
	\int_{\gamma}^{a}f(x)dx=\int_{-\infty}^{-b}f(x)dx,
	$$
\refeq{IntEq2} is equivalent to:
	$$
	n_{1}\int_{-\infty}^{-b}f(x)dx=(n_{1}+n_{2})\int_{0}^{a}f(x)dx.
	$$
Substituting: $s=\dfrac{1}{x}$, $c_{1}=\dfrac{1}{\gamma}$, $c_{2}=-\dfrac{1}{b}$, $c_{3}=0$, $c_{4}=\dfrac{1}{a}$ (see Figure \ref{fig:case1}), we get that
	$$
	n_{1}\int_{-1/b}^{0}\tilde{f}(s)ds=(n_{1}+n_{2})\int_{1/a}^{\infty}\tilde{f}(s)ds
	$$
	is equivalent to
	\begin{equation}\label{TagEq}
	n_{1}\int_{c_{2}}^{c_{3}}\tilde{f}(s)ds=(n_{1}+n_{2})\int_{c_{4}}^{\infty}\tilde{f}(s)ds,
	\end{equation}
where $\tilde f(s)ds$ is obtained from $f(x)dx$ by the substitution.
\end{proof}

In particular, for $n=3$, if the caustic $\C_{\gamma}$ is an ellipse with $\gamma<0$, then $n_{1}=1$.
Such polynomials and corresponding partitions $(3,1)$ do not arise in the study of Euclidean billiard trajectories.
On the other hand, if the caustic $\C_{\gamma}$ is an ellipse with $\gamma>0$, we have $n_{1}=2$.
Such polynomials for $\gamma>0$ can be explicitly expressed in terms of the Zolatarev polynomials, see Proposition \ref{propZolot}. Since their partition is $(3,2)$, they appeared before in the Euclidean case (see \cite{DragRadn2019rcd}).
The corresponding extremal polynomials $\hat{p}_{3}$ in both cases $\gamma<0$ and $\gamma>0$ are shown in Figure \ref{fig:ana5}.
We will provide in Proposition \ref{prop} the explicit formulae for such polynomials in terms of the general Akhiezer polynomials.
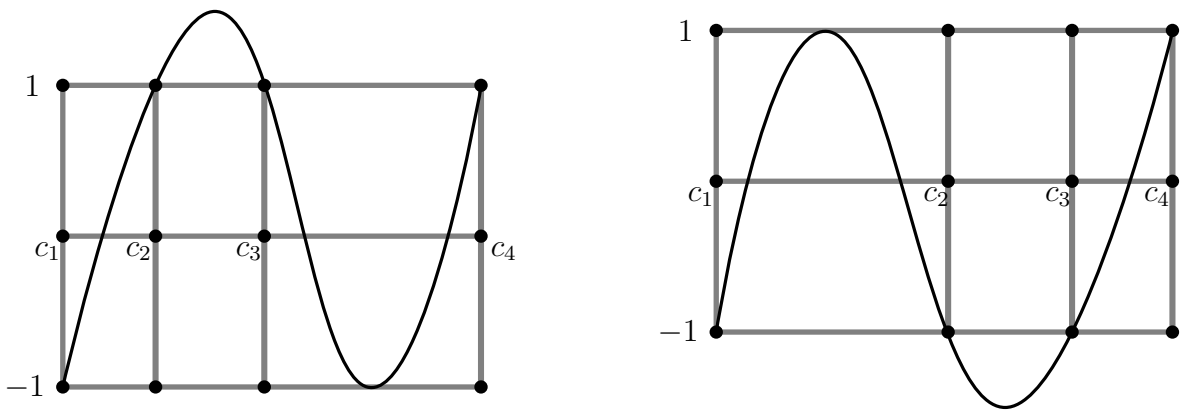
\begin{figure}[h]
	\begin{minipage}{0.5\textwidth}
		\begin{tikzpicture}[scale=1]
		
		\draw[line width=.7mm, gray](0,2)--(5.5,2);
		\draw[line width=.7mm, gray](0,0)--(5.5,0);
		\draw[line width=.7mm, gray](0,-2)--(5.5,-2);
		\draw[line width=.7mm, gray](0,2)--(0,-2);
		\draw[line width=.8mm, gray](1.22,-2)--(1.22,2);
		\draw[line width=.8mm, gray](2.65,-2)--(2.65,2);
		\draw[line width=.8mm, gray](5.5,-2)--(5.5,2);
		\draw[gray](0,-2)--(5.5,-2);
		

		\draw [very thick] plot [smooth, tension=.8] coordinates { (0,-2) (2,2.98) (4,-2) (5.5,1.98)};

		\fill[black] (2.65,-2) circle (2.5pt);
		\fill[black] (2.65,0) circle (2.5pt);
		\fill[black] (2.65,2) circle (2.5pt);
		\fill[black] (1.22,2) circle (2.5pt);
		\fill[black] (1.22,0) circle (2.5pt);
		\fill[black] (1.22,-2) circle (2.5pt);
		
		\fill[black] (0, 0) circle (2.5pt);
		\fill[black] (0,2) circle (2.5pt);
		\fill[black] (0,-2) circle (2.5pt);
		
		\fill[black] (5.5, 0) circle (2.5pt);
		\fill[black] (5.5,2) circle (2.5pt);
		\fill[black] (5.5,-2) circle (2.5pt);
		
		\draw (-.4,2) node {$1$};
		\draw (-.5,-2) node {$-1$};
		\draw (-.2,-.2) node {${c}_{1}$};
		\draw (1,-.2) node {${c}_{2}$};
		\draw (2.45,-.2) node {${c}_{3}$};
		\draw (5.8,-.2) node {${c}_{4}$};
		\end{tikzpicture}
	\end{minipage}
	\begin{minipage}{0.5\textwidth}
		\begin{tikzpicture}[scale=1]

		\draw[line width=.7mm, gray](0,2)--(6,2);
		\draw[line width=.7mm, gray](0,0)--(6,0);
		\draw[line width=.7mm, gray](0,-2)--(6,-2);
		\draw[line width=.7mm, gray](0,2)--(0,-2);
		\draw[line width=.8mm, gray](6,-2)--(6,2);
		\draw[line width=.8mm, gray](3.05,-2)--(3.05,2);
		\draw[line width=.8mm, gray](4.68,-2)--(4.68,2);
		\draw[gray](0,-2)--(6,-2);
		

		\draw [very thick] plot [smooth, tension=.8] coordinates { (0,-2) (1.5,1.98) (3.8,-3) (6,1.98)};

		\fill[black] (3.05,-2) circle (2.5pt);
		\fill[black] (3.05,0) circle (2.5pt);
		\fill[black] (3.05,2) circle (2.5pt);
		
		\fill[black] (4.68,-2) circle (2.5pt);
		\fill[black] (4.68,0) circle (2.5pt);
		\fill[black] (4.68,2) circle (2.5pt);
		
		\fill[black] (0, 0) circle (2.5pt);
		\fill[black] (0,2) circle (2.5pt);
		\fill[black] (0,-2) circle (2.5pt);
		
		\fill[black] (6, 0) circle (2.5pt);
		\fill[black] (6,2) circle (2.5pt);
		\fill[black] (6,-2) circle (2.5pt);
		
		\draw (-.4,2) node {$1$};
		\draw (-.5,-2) node {$-1$};
		\draw (-.2,-.2) node {${c}_{1}$};
		\draw (2.9,-.2) node {${c}_{2}$};
		\draw (4.5,-.2) node {${c}_{3}$};
		\draw (5.8,-.2) node {${c}_{4}$};

		\end{tikzpicture}
	\end{minipage}
	\caption{On the left: the polynomial $\hat{p}_{3}$ corresponding to $n=3$, $n_1=1$, $n_2=2$, $\gamma<0$. On the right: the polynomial $\hat{p}_{3}$ corresponding to $n=3$, $n_1=2$, $n_2=1$, $\gamma>0$.
	}
	\label{fig:ana5}
\end{figure}

Let us recall that the Chebyshev polynomials $T_n(x), n= 0, 1, 2,\dots$
defined by the recursion:
\begin{equation}\label{eq:cheb1}
T_0(x)=1, \quad T_1(x)=x,\quad T_{n+1}(x)+T_{n-1}(x)=2xT_n(x),
\end{equation}
for $n=1, 2\dots$ can be parameterized as
\begin{equation}\label{eq:cheb2}
T_n(x)=\cos n\phi,\quad x=\cos\phi,
\end{equation}
or, alternatively:
\begin{equation}\label{eq:cheb3}
 T_n(x)=\frac{1}{2}\left(v^n+\frac{1}{v^n}\right), \quad x=\frac{1}{2}\left(v+\frac{1}{v}\right).
\end{equation}
Denote $L_0=1$ and $L_n=2^{1-n}, n=1, 2,\dots$.
Then the Chebyshev Theorem states that the polynomials $L_nT_n(x)$ are characterized as the solutions of the following minmax problem:

{\it  find the polynomial of degree $n$ with the leading coefficient equal 1 which minimizes the uniform norm on the interval $[-1, 1]$.}

\subsection{Zolotarev polynomials}

Following the ideas of Chebyshev, his student Zolotarev posed and solved a handful of problems, including the following (\cite{AhiezerAPPROX,DragRadn2019rcd}):

{\it For the given real parameter $\sigma$ and all polynomials of degree $n$ of the form:
\begin{equation}\label{eq:zol1}
p(x)=x^n-n\sigma x^{n-1} + p_2x^{n-2}+\dots p_n,
\end{equation}
find the one with the minimal uniform norm on the interval $[-1, 1]$.}

Denote this minimal uniform norm as $L_n=L(\sigma, n)$.

For $\sigma>\tan^2(\pi/2n)$, the solution $z_n$ has the following property (\cite{AhiezerAPPROX}, p. 298, Fig. 9):

{\it $\Pi1$ -- The equation $z_n(x)=L_n$ has $n-2$ double solutions in the open interval $(-1, 1)$ and simple solutions at $-1, 1, \alpha, \beta$, where $1<\alpha <\beta$, while in the union of the intervals $[-1,1]\cup [\alpha, \beta]$ the inequality $z_n^2\le L_n$ is satisfied and $z_n>L_n$ in the complement.}

The polynomials $z_n$ are given by the following explicit formulae:
\begin{equation}\label{eq:zn}
z_n=\ell_n\left(v(u)^n+\frac1{v(u)^n}\right),
\quad
 x=\frac{\sn^2u +\sn^2\frac{K}{n}}{\sn^2u -\sn^2\frac{K}{n}},
\end{equation}
where
$$\ell_n=\frac1{2^n}\left(\frac{\sqrt{\kappa}\theta_1^2(0)}{H_1\left(\frac{K}{n}\right)\theta_1\left(\frac{K}{n}\right)}\right)^{2n}, \quad v(u)=\frac{H\left(\frac{K}{n}-u\right)}{H\left(\frac{K}{n}+u\right)}$$
and
$$\sigma=\frac{2\sn\frac{K}{n}}{\cn\frac{K}{n}\dn\frac{K}{n}}\left(\frac1{\sn\frac{2K}{n}}-\frac {\theta'\left(\frac{K}{n}\right)}{\theta\left(\frac{K}{n}\right)}\right)-1.
$$
Formulae for the endpoints of the second interval are
\begin{equation}\label{eq:alphabetan}
\alpha =\frac{1+\kappa^2\sn^2\frac{K}{n}}{\dn^2\frac{K}{n}}, \quad \beta =\frac{1+\sn^2\frac{K}{n}}{\cn^2\frac{K}{n}},
\end{equation}
with
$$
\kappa^2=\frac{(\alpha-1)(\beta+1)}{(\alpha+1)(\beta-1)}.
$$
According to Cayley's condition for $n=3$ and $\gamma \in (0, a)$ we have
$$
\gamma=\frac{ab(a-b)+2ab\sqrt{a^2+ab+b^2}}{(a+b)^2}.
$$
In order to derive the formulas for $\hat p_3$ in terms of $z_3$, let us construct
an affine transformation:
$$
h:[-1, 1]\cup [\alpha, \beta]\rightarrow [-b^{-1}, 0]\cup[a^{-1}, \gamma^{-1}],
\quad
 h(x) = \hat a x +\hat b.
$$

We immediately get
$$
\hat a = -\hat b, \quad \hat a = \frac1{2b}
$$
and
\begin{equation}\label{eq:alphabetacayley3}
\alpha = 2t+1,
\end{equation}
\begin{equation}\label{eq:alphabetacayley4}
\gamma=\frac{2b}{\beta -1}
\end{equation}
where $t=b/a$.

Now we get the following
\begin{proposition}\label{propZolot}
The polynomial $\hat p_3$ can be expressed through the Zolotarev polynomial $z_3$ up to a nonessential constant factor:
$$
\hat p_3 (s) \sim z_3( 2bs+1).
$$
\end{proposition}

To verify the proposition, we should  certify that the definition of  $\alpha$ and $\beta$ from \refeq{eq:alphabetan} for $n=3$ and the relations \refeq{eq:alphabetacayley3}, \refeq{eq:alphabetacayley4} are compatible with the formula for $\gamma$ we got before as Cayley condition, see \refeq{CaleyEq1}.

In order to do that we will use well-known identities for the Jacobi elliptic functions:

\begin{gather}
\sn^2u+\cn^2u=1,
\\
\kappa^2\sn^2u+\dn^2u=1,
\\
\sn(u+v)=\frac{\sn\, u\cn\, v\dn \,v + \sn\, v\cn\, u\dn\, u}{1-\kappa^2\sn^2 u \sn ^2 v},
\\
\sn (K-u)=\frac{\cn\, u}{\dn\, u}.
\end{gather}
In particular, we get
\begin{gather}
\sn \left(\frac {2K}{3}\right)= \frac{2\sn\, \frac{K}{3}\cn \,\frac{K}{3}\dn \,\frac{K}{3}}{1-\kappa^2\sn^4\frac{K}{3}},
\\
\sn \left(\frac {2}{3}K\right)= \sn \left(K-\frac {K}{3}\right)=\frac{\cn \,\frac{K}{3}}{\dn\, \frac{K}{3}}.
\end{gather}
Let us denote
$$Y=\sn \left(\frac {K}{3}\right),$$
then from the previous two relations we get as in (\cite{DragRadn2019rcd}):
$$
1-2Y+2\kappa^2Y^3 -\kappa^2Y^4=0.
$$
We can express $\kappa$ in terms of $Y$ and get:
\begin{equation}\label{Add F}
\kappa^2=\frac{2Y-1}{Y^3(2-Y)}.
\end{equation}
By plugging the last relation into \refeq{eq:alphabetan} for $n=3$ we get
$$
\alpha=\frac{Y^2-4Y+1}{Y^2-1}.
$$
Since, at the same time from the Cayley condition we have $\alpha = 2t+1$, with $t=b/a$, we can express $Y$ in terms of $t$:
$$
tY^2+2Y-(t+1)=0,
$$
and
\begin{equation}\label{eq:Yt}
Y=\frac{-1\pm \sqrt{1+t+t^2}}{t}.
\end{equation}
We plug the last relation into the formula for $\beta$ from \refeq{eq:alphabetan} for $n=3$
$$
\beta=\frac{1+Y^2}{1-Y^2},
$$
and we get another formula for $\beta$ in terms of $t$:
\begin{gather}\label{eq:betat}
\beta=\frac{2t^2+t+2-\pm 2\sqrt{t^2+t+1}}{-t-2 \pm 2\sqrt{t^2+t+1}}.
\end{gather}
We see that the last formula with the choice of the $+$ sign corresponds to a formula for $\beta$ from \refeq{eq:alphabetacayley4}. This formula relates $\beta$ and $\gamma$  from the Caley condition \refeq{CaleyEq1}.
From \refeq{eq:betat}, taking the positive sign in $\beta$ yields,
\begin{gather}\label{beta2}
\beta=\frac{2t^2+t+2- 2\sqrt{t^2+t+1}}{-t-2 + 2\sqrt{t^2+t+1}}.
\end{gather}
Substituting \refeq{beta2} into \refeq{eq:alphabetacayley4} produces
\begin{equation}\label{P1}
\gamma=\frac{2b}{\beta -1} = b\dfrac{-2-t+2\sqrt{1+t+t^2}}{2+t+t^2-2\sqrt{1+t+t^2}}.
\end{equation}
On the other hand, from the Cayley formula \refeq{CaleyEq1}:
$$
\gamma= \dfrac{ab}{(a+b)^{2}}(a-b+2\sqrt{a^{2}+ab+b^{2}})
$$
Knowing that $t=\frac{b}{a}$, the equation is equivalent to
\begin{equation}\label{P2}
\gamma= b\dfrac{1-t+2\sqrt{1+t+t^2}}{(1+t)^2}
\end{equation}
In order to show that the two expressions in \refeq{P1} and \refeq{P2} are identical, we simplify their difference that yields zero. This finalizes the verification.
(One can observe that the $-$ sign option from the formula \refeq{eq:betat} would correspond to
the $-$ sign in the formula for $\gamma$ \refeq{CaleyEq1}.

 Among the polynomials $\hat p_n$ the property of type $\Pi1$ can be attributed
only to those with $n=2k+1$ and winding numbers $(2k+1, 2k)$, in other words to those with the signature $(0, 2k-1)$.

\subsection{Akhiezer polynomials on symmetric intervals $[-1, -\alpha]\cup[\alpha, 1]$}

The problem of finding polynomials of degree $n$ with the leading coefficient 1 and minimizing the uniform norm on the union of two symmetric intervals $[-1, -\alpha]\cup[\alpha, 1]$, for given $0<\alpha <1$ appeared to be of a significant interest in radio-techniques applications. Following the ideas of Chebyshev and Zolotarev, Akhiezer derived in 1928 the explicit formulae for such polynomials $A_n(x;\alpha)$ with the deviation $L_n(\alpha)$ \cite{AhiezerAPPROX,Akh4}.

These formulas are especially simple in the case of even degrees $n=2m$, when Akhiezer polynomials $A_{2m}$ are obtained by a quadratic substitution from the Chebyshev polynomial $T_m$:

\begin{equation}\label{eq:A2m}
A_{2m}(x;\alpha)=\frac{(1-\alpha^2)^m}{2^{2m-1}}T_m\left(\frac{2x^2-1-\alpha^2}
{1-\alpha^2}\right),
\end{equation}
with
$$
L_{2m}(\alpha)=\frac{(1-\alpha^2)^m}{2^{2m-1}}.
$$
We are going to construct $\hat p_4(s)$ up to a nonessential constant factor as a composition of $A_4(x;\alpha)$ for certain $\alpha$ and an affine transformation.
We are going to study the possibility to have an affine transformation
$$
g:[-1,-\alpha]\cup[\alpha, 1]\rightarrow [-b^{-1}, \gamma^{-1}]\cup [0, a^{-1}],\quad g(x)=\hat a x +\hat b,
$$
which corresponds to the case when $\gamma<-b$ ie $a>b$. For $n=4$ such caustic is \refeq{4pCaustic}:
$$
\gamma=\frac{ab}{b-a}.
$$

From $g(-1)=-b^{-1}$, $g(1)=a^{-1}$
we get
$$
\hat a = \frac{a+b}{2ab}, \quad \hat b =\frac{b-a}{2ab}.
$$
Then, from $g(\alpha)=0$ we get
$$
\alpha =\frac {a-b}{a+b}.
$$
Finally, we calculate:
$$
g(-\alpha)=\frac{a+b}{2ab}\frac {b-a}{a+b} + \frac{b-a}{2ab}=\frac{b-a}{ab}.
$$
We recognize $\gamma^{-1}$ on the righthand side of the last relation.\newline
This proves the following:
\begin{proposition} In this case the polynomial $\hat p_4(s)$ equals, up to a constant multiplier, to
	\begin{equation}\label{eq:p4t2}
	\hat p_4(s) \sim T_2(2abs^2 +2(a-b)s+1),
	\end{equation}
	where $T_2(x)=2x^2-1$ is the second Chebyshev polynomial and $x=\frac{1}{a+b}\big(2abs+a-b\big)$.
\end{proposition}
Let us study the possibility to have an affine transformation
$$
f:[-1,-\alpha]\cup[\alpha, 1]\rightarrow [-b^{-1}, 0]\cup [\gamma^{-1}, a^{-1}],\quad f(x)=\hat a x +\hat b,
$$
which corresponds to the case when $\gamma>a$ ie $a<b$. For $n=4$ such caustic is
$$
\gamma=\frac{ab}{b-a}.
$$

From $f(-1)=-b^{-1}$, $f(1)=a^{-1}$
we get
$$
\hat a = \frac{a+b}{2ab}, \quad \hat b =\frac{b-a}{2ab}.
$$
Then, from $f(-\alpha)=0$ we get
$$
\alpha =\frac {b-a}{a+b}.
$$
Finally, we calculate:
$$
f(\alpha)=\frac{a+b}{2ab}\frac {b-a}{a+b} + \frac{b-a}{2ab}=\frac{b-a}{ab}.
$$
We recognize $\gamma^{-1}$ on the righthand side of the last relation.\\
This proves the following proposition which is the same as \refeq{eq:p4t2}.
\begin{proposition} In this case the polynomial $\hat p_4(s)$ is equal up to a constant multiplier to
	\begin{equation}\label{eq:p4t3}
	\hat p_4(s) \sim T_2(2abs^2 +2(a-b)s+1),
	\end{equation}
	where $T_2(x)=2x^2-1$ is the second Chebyshev polynomial and $x=\dfrac{1}{a+b}\big(2abs+a-b\big)$.
\end{proposition}
Let us study the possibility to have an affine transformation
$$
h:[-1,-\alpha]\cup[\alpha, 1]\rightarrow [\gamma^{-1}, -b^{-1}]\cup  [0,a^{-1}],\quad h(x)=\hat a x +\hat b,
$$
which corresponds to the case when $\gamma \in (-b,0)$. For $n=4$ such caustic is
$$
\gamma=-\frac{ab}{a+b}.
$$

From $h(1)=a^{-1}$, $h(\alpha)=0$
we get
$$
\hat a = \frac{1}{1-\alpha}\frac{1}{a}, \quad \hat b =-\frac{\alpha}{1-\alpha}\frac{1}{a}.
$$
Then, from $h(-\alpha)=-\frac{1}{b}$ we get
$$
\frac{\alpha}{1-\alpha} =\frac {a}{2b}.
$$
ie
$$
\alpha =\frac {a}{a+2b}.
$$
Finally, we calculate:
$$
h(-1)=-\big(1+\frac{\alpha}{1-\alpha}\big)\frac{1}{a}-\frac{\alpha}{1-\alpha}\frac{1}{a},
$$
$$
h(-1)=-\big(1+\frac {a}{2b}\big)\frac{1}{a}-\frac {a}{2b}\frac{1}{a}=-\frac{1}{a}-\frac{1}{b}= -\frac{a+b}{ab}.
$$
We recognize $\gamma^{-1}$ on the righthand side of the last relation.\\
This proves the following:
\begin{proposition} In this case the polynomial $\hat p_4(s)$ equal,s up to a constant multiplier, to
	\begin{equation}\label{eq:p4t4}
	\hat p_4(s) \sim T_2\left(\frac{2a^{2}bs^{2}+2a^{2}s-(a+b)}{a+b}\right),
	\end{equation}
	where $T_2(x)=2x^2-1$ is the second Chebyshev polynomial.
\end{proposition}
Let us study the possibility to have an affine transformation
$$
l:[-1,-\alpha]\cup[\alpha, 1]\rightarrow [-b^{-1},0]\cup  [a^{-1},\gamma^{-1}],\quad l(x)=\hat a x +\hat b,
$$
which corresponds to the case when $\gamma \in (0,a)$. For $n=4$ such caustic is
$$
\gamma=\frac{ab}{a+b}.
$$

From $l(-1)=-b^{-1}$, $l(-\alpha)=0$
we get
$$
\hat a = \frac{1}{1-\alpha}\frac{1}{b}, \quad \hat b =\frac{\alpha}{1-\alpha}\frac{1}{b}.
$$
Then, from $l(\alpha)=\frac{1}{a}$ we get
$$
\frac{\alpha}{1-\alpha} =\frac {b}{2a}.
$$
ie
$$
\alpha =\frac {b}{b+2a}.
$$
Finally, we calculate:
$$
l(1)=\big(1+\frac{\alpha}{1-\alpha}\big)\frac{1}{b}+\frac{\alpha}{1-\alpha}\frac{1}{b},
$$
$$
l(1)=\big(1+\frac {b}{2a}\big)\frac{1}{b}+\frac {b}{2a}\frac{1}{b}= \frac{1}{a}+\frac{1}{b}=\frac{a+b}{ab}.
$$
We recognize $\gamma^{-1}$ on the righthand side of the last relation.

This proves the following:

\begin{proposition} In this case the polynomial $\hat p_4(s)$ equals, up to a constant multiplier, to
	\begin{equation}\label{eq:p4t5}
	\hat p_4(s) \sim T_2\left(\frac{2ab^{2}s^{2}-2b^{2}s-(a+b)}{a+b}\right),
	\end{equation}
	where $T_2(x)=2x^2-1$ is the second Chebyshev polynomial.
\end{proposition}

\subsection{General Akhiezer polynomials on unions of two intervals}
Following Akhiezer \cites{Akh1,Akh2,Akh3}, let us consider the union of two intervals $[-1,\alpha]\cup [\beta, 1]$, where
\begin{equation}\label{Akhiezer}
\alpha = 1-2\sn^2(\frac{m}{n}K), \quad \beta = 2\sn^2\left(\frac{n-m}{n}K\right)-1.
\end{equation}
Define
\begin{equation}\label{AKhTAn}
TA_{n}(x,m,\kappa)=L\left(v^{n}(u)+\frac{1}{v^{n}(u)}\right),
\end{equation}
where
$$
v(u)=\dfrac{\theta_1\big(u-\frac{m}{n}K\big)}{\theta_1\big(u+\frac{m}{n}K\big)},
\quad
x=\dfrac{\sn^2(u)\cn^2(\frac{m}{n}K)+\cn^2(u)\sn^2(\frac{m}{n}K)}{\sn^2(u)-\sn^2(\frac{m}{n}K)},
$$
and
$$
L=\frac{1}{2^{n-1}}\left(\dfrac{\theta_0(0)\theta_3(0)}{\theta_0(\frac{m}{n}K)\theta_3(\frac{m}{n}K)}\right),
 \quad
  \kappa^2=\frac{2(\beta -\alpha)}{(1 -\alpha)(1+\beta)}.
$$
Here $\theta_i,\, i=0. 1, 2, 3,$ denote the standard Riemann theta functions, see for example \cite{Akh4} for more details.
Akhiezer proved the following result:

\begin{theorem}[Akhiezer]\label{th:Akhiezer}
	\begin{itemize}
		\item [(a)] The function $TA_{n}(x,m,\kappa)$ is a polynomial  of degree $n$ in $x$ with the leading coefficient $1$ and the second coefficient equal to $-n\tau_{1}$, where
		$$
		\tau_{1}=-1+2\dfrac{\sn(\frac{m}{n}K)\cn(\frac{m}{n}K)}{\dn(\frac{m}{n}K)}\left(\frac{1}{sn(\frac{2m}{n}K)}-\frac{\theta_0^{\prime}(\frac{m}{n}K)}{\theta_0(\frac{m}{n}K)}\right).
		$$
		\item[(b)]
		The maximum of the modulus of $TA_{n}$ on the union of the two intervals $[-1,\alpha]\cup [\beta, 1]$ is $L$.
		\item[(c)]
		The function $TA_{n}$ takes values $\pm L$ with alternating signs at $\mu=n-m+1$ consecutive points of the interval $ [- 1, \alpha]$ and at $\nu=m+1$ consecutive points of the interval $ [\beta, 1]$. In addition
		$$
		TA_{n}(\alpha,m,\kappa)=TA_{n}(\beta,m,\kappa)=(-1)^{m}L,
		$$
		and for any $x\in (\alpha, \beta)$, it holds:
		$$
		(-1)^{m}TA_{n}(x,m,\kappa)>L.
		$$
		\item[(d)]
		Let $F$ be a polynomial of degree $n$ in $x$ with the leading coefficient $1$, such that:
	\end{itemize}
	\begin{itemize}
		\item [i.)]
		$max|F(x)|=L$ for $x\in [-1,\alpha]\cup [\beta, 1]$;
		\item[ii)]
		F(x) takes values $\pm L$ with alternating signs at n-m+1 consecutives points of the interval
		$[-1, \alpha]$ and at m+1 consecutive points of the interval $[\beta, 1]$.
	\end{itemize}
	Then $F(x)=TA_{n}(x,m,\kappa).$
\end{theorem}
Let us determine the affine transformations when the caustic is an ellipse.

\paragraph*{Case $\gamma \in (-b,0)$.}

For
$h:[-1,\alpha]\cup[\beta, 1]\rightarrow [\gamma^{-1},-b^{-1}]\cup  [0,a^{-1}]$, $h(x)=\hat a x +\hat b$,
we get
$$
\hat a = \frac{1}{\beta-\alpha}\frac{1}{b},
\quad
\hat b =\frac{-\beta}{\beta-\alpha}\frac{1}{b},
\quad
\frac{1-\beta}{\beta -\alpha}=\frac{b}{a}.
$$
Thus:

\begin{equation}\label{GenAkhiezerTrans}
\gamma=\frac{\beta-1}{1+\beta}a=\frac{ \alpha-\beta}{\beta+1}b
\end{equation}

\begin{example}\label{ex:n=3,m=2}
	For $n=3$ and $m=2$. From \refeq{Akhiezer}, one gets:
	$$
	\alpha = 1-2\sn^2\frac{2}{3}K, \quad \beta = 2\sn^2\frac{K}{3}-1.
	$$
		It follows that:
	\begin{equation}\label{Eq1}
	\frac{b}{a}=t= \frac{1-\beta}{\beta-\alpha}= \frac{1-\sn^2\frac{K}{3}}{\sn^2\frac{2}{3}K
		+\sn^2\frac{K}{3}-1},
	\end{equation}
	Thus
	\begin{equation}\label{GenAkhiezerTrans1}
	\gamma=b\frac{\alpha -\beta}{\beta+1}=b\frac{1-\sn^2\frac{K}{3}-sn^2\frac{2}{3}K}{\sn^2\frac{K}{3}}.
	\end{equation}
	From the addition formula:
	$$
	\sn\frac{2}{3}K=\sn(K-\frac{K}{3})=\frac{\sn K \cn\frac{-K}{3} \dn\frac{-K}{3}+\sn\frac{-K}{3}\cn K \dn K}{1-\kappa^2\sn^2\frac{-K}{3}\sn^2K}.
	$$
	Hence
	$$
	\sn^2\frac{2}{3}K=\frac{1-\sn^2\frac{K}{3}}{1-\kappa^2\sn^2\frac{K}{3}}
	=
	\frac{\sn^2\frac{2}{3}K-1}{\kappa^2\sn^2\frac{2}{3}K-1}.
	$$
	Let $sn\frac{K}{3}=Z$:
	\begin{equation*}
	\kappa^{2}=\frac{2Z-1}{Z^3(2-Z)}.
	\end{equation*}
	Also
	$$
	\alpha = 2Z^2-4Z+1,
	\quad
	\beta =2Z^2-1.
	$$
	Equation \refeq{Eq1} implies that
	\begin{equation}
	t=\frac{1-Z^2}{2Z-1}.
	\end{equation}
	Thus, we have two expressions for $\gamma$. One is from the Cayley condition \refeq{CaleyEq1}
	and the other is from \refeq{GenAkhiezerTrans}. We want to show that these two expressions are identical, that is
	\begin{equation}\label{EquaLambda}
	b\frac{\alpha -\beta}{\beta+1}=-\dfrac{ab}{(a+b)^{2}}(-a+b+2\sqrt{a^{2}+ab+b^{2}}).
	\end{equation}
	In order to do so, we first express both the left hand side and the right hand side of the above equation in terms of $t=\dfrac{b}{a}$ and then transform both sides in terms of $Z$. We show that the left and the right handsides yield the same expression:
	$$
	\frac{Z^2-Z+1}{2Z-1},
	$$
	therefore \refeq{EquaLambda} holds.
\end{example}

\paragraph*{Case $\gamma \in (0,a)$.}

For
$
l:[-1,\alpha]\cup[\beta, 1]\rightarrow [-b^{-1},0]\cup  [a^{-1}, \lambda^{-1}]$,
$l(x)=\hat a x +\hat b$,
we get
$$
\hat a = \frac{1}{\alpha + 1}\frac{1}{b},
\quad
\hat b =\frac{-\alpha}{\alpha+1}\frac{1}{b},
\quad
\frac{\alpha +1}{\beta -\alpha}=\frac{a}{b}.
$$
Thus
\begin{equation}\label{Eq:LambaPositiv}
\gamma =\frac{\alpha +1}{1-\alpha}b
\end{equation}

\begin{example}
	For $n=3$, and $m=1$. From  \refeq{Akhiezer}, one gets:
	$$
	\alpha = 1-2\sn^2\frac{K}{3}, \quad \beta = 2\sn^2\frac{2K}{3}-1.
	$$	
	\begin{equation}\label{Eq2}
	\frac{b}{a}=t= \frac{\beta - \alpha}{\alpha+1}= \frac{\sn^2\frac{2}{3}K+\sn^2\frac{K}{3}-1}{1-\sn^2\frac{K}{3}}
	\end{equation}
	Thus
	$$
	\gamma=\frac{1-\sn^2\frac{K}{3}}{\sn^2\frac{K}{3}}b.
	$$
From the addition formula:
	$$
	\sn\frac{2K}{3}=\sn(K-\frac{K}{3})=\frac{\sn K \cn\frac{-K}{3} \dn\frac{-K}{3}+\sn\frac{-K}{3}\cn K \dn K}{1-\kappa^2\sn^2\frac{-K}{3}\sn^2K}.
	$$
	Hence
	$$
	\sn^2\frac{2}{3}K=\frac{1-\sn^2\frac{K}{3}}{1-\kappa^2\sn^2\frac{K}{3}}.
	$$
	
	Let $\sn\frac{K}{3}=Z$:
	\begin{equation*}
	\kappa^{2}=\frac{2Z-1}{Z^3(2-Z)}.
	\end{equation*}
	Also
	$
	\beta  = -2Z^2+4Z-1,
	$
	and
	$\alpha =1-2Z^2$.
	Equation \refeq{Eq2} implies that
	\begin{equation}
	t=-\frac{2Z-1}{Z^2-1},
	\end{equation}
	Thus, we have two expressions for $\gamma$. One is from the Cayley condition \refeq{CaleyEq1}
	and the other is from \refeq{Eq:LambaPositiv}.
	Similarly as in Example \ref{ex:n=3,m=2}, we can show that these two expressions are identical.
\end{example}

\begin{proposition}\label{prop}
	For $n=3$ and $\gamma \in (-b,0)$, the polynomial $\hat{p}_{3}$ is up to a nonessential factor equal to:
	$$
	\hat{p}_{3} ~\sim  TA_{3}\left(2a\left(1-\sn^2\frac{K}{3}\right)s+2\sn^2\frac{K}{3}-1 ; 2, \kappa \right),
	$$
	For $n=3$ and $\gamma \in (0,a)$, the polynomial $\hat{p}_{3}$ is up to a nonessential factor equal to:	
	$$
	\hat{p}_{3} ~\sim  TA_{3}\left(2b\left(1-\sn^2\frac{K}{3}\right)s+1-2\sn^2\frac{K}{3} ; 1, \kappa \right)
	$$
\end{proposition}
Now, using the Akhiezer Theorem part (c), see Theorem \ref{th:Akhiezer}, one can compare and see that the number of internal extremal points
coincides with $n_1-1$ and $n_2-1$ as proposed in Theorem \ref{th:impactwinding}. These numbers match with Figure \ref{fig:ana5} and the table from Section \ref{sec:examples-table}.

\section{Periodic light-like trajectories and Chebyshev polynomials}
\label{sec:ll}

The light-like billiard trajectories, by definition, have at each point the velocity $v$ satisfying $\langle v,v\rangle=0$.
Their caustic is the conic at infinity $\C_{\infty}$.
Since successive segments of such trajectories are orthogonal to each other, the light-like trajectories can close only after an even number of reflections.
In \cite{DragRadn2012adv}*{Theorem 3.3}, it is proved that a light-like billiard trajectory within $\E$ is periodic with even period $n$ if and only if
\begin{equation}\label{eq:arc}
\arccot\sqrt{\frac ab}
\in
\left\{ \frac{k\pi}n \mid 1\le k<\frac{n}2, \left(k,\frac{n}2\right)=1 \right\}.
\end{equation}
For $k$ not being relatively prime with $n/2$, the corresponding light-like trajectories are also periodic, and their period is a divisor of $n$.

Applyin the limit $\gamma\to+\infty$ in Corollary \ref{cor:pell-periodic}, we get
the following proposition.

\begin{proposition}\label{prop:ll}
A light-like trajectory within ellipse $\E$ is periodic with period $n=2m$ if and only if there exist real polynomials $\hat{p}_m(s)$ and $\hat{q}_{m-1}(s)$ of degrees $m$ and $m-1$ respectively if and only if:
\begin{itemize}
\item $\hat{p}_m^2(s)-\left(s-\dfrac1a\right)\left(s+\dfrac1b\right)\hat{q}_{m-1}^2(s)=1$; and
\item $\hat{q}_{m-1}(0)=0$.
\end{itemize}	
\end{proposition}

The first condition from Proposition \ref{prop:ll} is the standard Pell equation describing extremal polynomials on one interval $[-1/b,1/a]$, thus the polynomials $\hat{p}_m$ can be obtained as Chebyshev polynomials composed with an affine transformation $[-1/b,1/a]\to[-1,1]$.
The additional condition $\hat{q}_{m-1}(0)=0$, which is equivalent to $\hat{p}_m'(0)=0$ implies an additional constraint on parameters $a$ and $b$.
We have the following

\begin{proposition} The polynomials $\hat {p}_m$ and the parameters $a, b$ have the following properties:
\begin{itemize}
	\item
$\hat{p}_m(s)=T_m\left(\dfrac{2ab}{a+b}s+\dfrac{a-b}{a+b}\right)$, where $T_m$ is defined by \refeq{eq:cheb2};
\item
the condition $\hat{q}_{m-1}(0)=0$ is equivalent to \refeq{eq:arc}.
\end{itemize}
\end{proposition}
\begin{proof}
The increasing affine transformation $h:[-1/b, 1/a]\rightarrow [-1, 1]$ is given by the formula $h(s)=(2ab s + a-b)/(a+b)$.
The internal extremal points of the Chebyshev polynomial $T_m$ of degree $m$ on the interval $[-1, 1]$ are given by
$$
x_k=\cos \left(\frac{k}{m}\pi\right), \quad k=1, \dots, m-1,
$$
according to the formula \refeq{eq:cheb2}. The second item follows from $h(0)=x_k$. This is equivalent
to
$$
\frac{a-b}{a+b}\in\left\{\cos \left(\frac{k}{m}\pi\right)\mid k=1, \dots, m-1\right\},
$$
which is equivalent  to  \refeq{eq:arc}.
\end{proof}

\subsection*{Acknowledgment}
\addcontentsline{toc}{section}{Acknowledgment}
The research of V.~D.~and M.~R.~was supported by the Discovery Project \#DP190101838 \emph{Billiards within confocal quadrics and beyond} from the Australian Research Council and Project \#174020 \emph{Geometry and Topology of Manifolds, Classical Mechanics and Integrable Systems} of the Serbian Ministry of Education, Technological Development and Science.
V.~D.~would like to thank Sydney Mathematics Research Institute and their International Visitor Program for kind hospitality.

\end{document}